\documentclass[10pt]{amsart}
\usepackage{amsmath,amssymb,amsfonts} 
\usepackage{graphics}                 
\usepackage[square, numbers]{natbib}
\usepackage[english]{babel}
\usepackage{tikz} 
\usepackage[utf8]{inputenc}
\usepackage[all]{xy}	
\usepackage[T1]{fontenc}
\usepackage{pifont}
\usepackage{hyperref}
\usepackage[margin=1.2in]{geometry}

\newcommand{\U}{\mathcal{U}}

\newcommand{\R}{\mathbb{R}}

\theoremstyle{definition}
\newtheorem{defn}{Definition}[subsection]
\newtheorem{rem}[defn]{Remark}
\newtheorem{cons}[defn]{Construction}

\newtheorem{nota}[defn]{Notation}

\theoremstyle{plain}
\newtheorem{Prop}[defn]{Proposition}
\newtheorem{Thm}[defn]{Theorem}
\newtheorem{cor}[defn]{Corollary}
\newtheorem{lem}[defn]{Lemma}

\newtheorem*{asser}{Claim}

\newtheorem*{conj}{Conjecture}

\theoremstyle{exampstyle}
\newtheorem{thmx}{Theorem}

\newtheorem{corx}{Corollary}

\theoremstyle{exampstyle}

\def\Ind{\setbox0=\hbox{$x$}\kern\wd0\hbox to 0pt{\hss$\mid$\hss}
\lower.9\ht0\hbox to 0pt{\hss$\smile$\hss}\kern\wd0}
\def\Notind{\setbox0=\hbox{$x$}\kern\wd0\hbox to 0pt{\mathchardef
\nn=12854\hss$\nn$\kern1.4\wd0\hss}\hbox to
0pt{\hss$\mid$\hss}\lower.9\ht0 \hbox to
0pt{\hss$\smile$\hss}\kern\wd0}

\newtheoremstyle{exampstyle}
{3pt} 
{3pt} 
{\itshape} 
{} 
{\bfseries} 
{.} 
{.5em} 
{} 

\makeindex

\title{Generic planar algebraic vector fields are strongly minimal and disintegrated}
\author{Rémi Jaoui}
\thanks{Partially supported by NSF grant DMS-1760212}
\address{Rémi Jaoui, Department of Mathematics, University of Notre Dame}
\email{rjaoui@nd.edu}
\date\today
\begin{document}

\maketitle
\begin{abstract}
In this article, we study model-theoretic properties of algebraic differential equations of order $2$, defined over constant differential fields. In particular, we show that the set of solutions of a ``general'' differential equation of order $2$ and of degree $d \geq 3$ in a differentially closed field is strongly minimal and disintegrated.

We also give two other formulations of this result in terms of algebraic (non)-integrability and algebraic independence of the analytic solutions of a general planar algebraic vector field.
\end{abstract}

\textbf{A question of Poizat and planar vector fields.}  The study of disintegrated differential equations originated from questions concerning properties of the models of the (complete) theory $\textbf{DCF}_0$ of existentially closed differential fields of characteristic $0$ such as the existence of minimal models and the number of non-isomorphic models of a given uncountable cardinality.

Both of these questions were solved in the 70's by Shelah in \cite{She3} and Rosenlicht in \cite{Ros} while providing evidence for the existence of large families of pairwise orthogonal autonomous differential equations of \textit{order one} with trivial forking geometry. While the differential algebraic methods of \cite{Ros} were specific to the order one case, it was expected that a similar picture should hold for differential equations of higher order, namely: 

\begin{itemize}
\item[(1)] A ``sufficiently general'' differential equation of order $n \geq 2$ defines a strongly minimal set of the theory $\textbf{DCF}_0$.
\item[(2)] A ``sufficiently general'' differential equation of order $n \geq 2$ is disintegrated. In other words, its set of solutions has trivial forking geometry.
\end{itemize}

The first point is already suggested in \cite{She3} and appears explicitly in \cite{Poi3}.  Here, the word ``general'' is understood in an algebraic fashion: namely, a differential equation is general if its coefficients --- when viewed in a sufficiently large (in particular non linear) family of differential equations ---  bear no particular differential algebraic relations. 

The second point  formulated in \citep[pp560]{Poi2} expresses that every algebraic dependence between solutions of a ``sufficiently general'' differential equation comes from  a combination of binary algebraic dependencies (disintegration property). The same property can also be phrased as a degeneracy of the combinatorial geometry associated to this strongly minimal set, hence the terminology of trivial forking geometry.   It is worth mentioning at this point that ---  in contrast with the order one case --- disintegration (or trivial forking geometry) does not imply $\omega$-categoricity in general: namely, there are disintegrated strongly minimal sets $X$ such that for no definable equivalence relation $E$ on $X$ with finite classes, $X/E$ is trivial i.e. a pure infinite set with no structure at all (see \cite{Frei} and \cite{Cas}).

During the nineties, an impressive classification of the \textit{non disintegrated minimal types} of the theory $\textbf{DCF}_0$ was achieved by Hrushovski and Sokolovic in \cite{Sok}, strengthening even more the interest towards problems (1) and (2): while a powerful ``structure theory'' is available for non-disintegrated minimal types, general differential equations should lie on the other side of the picture, justifying furthermore that disintegration and minimality are prominent phenomena for algebraic differential equations.

In this article, we settle the problem for planar algebraic vector fields and  show that the differential equation  (of order two) associated to a general planar algebraic vector field of degree $d \geq 3$  satisfies the conclusions of  (1) and (2):

{\thmx[Model-theoretic version]\label{introa} Let $d \geq 3$ and let $v$ be a planar algebraic vector field of degree $d$ with $\mathbb{Q}$-algebraically independent coefficients.

The set of solutions of the differential equation $(\mathbb{A}^2,v)$ in a differentially closed field is strongly minimal and disintegrated.} 
\vspace{0.2cm}

Fixing $d \geq 1$, a complex planar vector field of degree $\leq d$ is represented as:
$$ v(x,y) = f(x,y) \frac \partial {\partial x}  + g(x,y) \frac \partial {\partial y} 
\text{ with } f(x,y),g(x,y) \in \mathbb{C}[x,y] \text{ of degree }\leq d.$$ So we can view the family of complex planar algebraic vector fields of degree $\leq d$ as an $\emptyset$-definable family of vector fields $(v_s, s \in S(\mathbb{C}))$ parametrized by an affine space $S = \mathbb{A}^N$ representing the coefficients of the polynomials $f(x,y)$ and $g(x,y)$. Theorem \ref{introa} says that for $d \geq 3$, the property: 
\begin{center}
$\mathcal P(s)$: the set of solutions of $(\mathbb{A}^2, v(s))$ in a differentially closed field is strongly minimal and disintegrated.
\end{center}
holds for generic elements (over $\emptyset$) of $S(\mathbb{C})$, namely for vector fields $v(s)$ of degree $d$ and $\mathbb{Q}$-algebraically independent coefficients (in particular no coefficient of $v(s)$ is allowed to vanish).

For $d = 1$, it is easy to see that the conclusion of the theorem fails but I don't know if Theorem \ref{introa} holds for $d = 2$. The hypothesis $d \geq 3$ comes from the fact if one replaces the affine plane $\mathbb{A}^2$ by the affine line $\mathbb{A}^1$, Theorem \ref{introa} holds for $d \geq 3$ but not for $d = 2$ (see \cite{Ros}).

It is worth mentioning that for any $\emptyset$-definable family as above, using standard arguments involving that the field of constants of a differentially closed field is an algebraically closed field with no additional structure, we show in Section 3.3 that the following are equivalent:

\begin{itemize}
\item[(i)] The property $\mathcal P(s)$ holds on a measurable set of parameters $s \in S(\mathbb{C})$ of positive Lebesgue measure.
\item[(ii)] The property $\mathcal P(s)$ holds on a somewhere dense $G_\delta$-set of parameters $s \in S(\mathbb{C})^{an}$: the property $\mathcal P(s)$ holds on a countable intersection $A = \bigcap_{i \in \mathcal I} U_i$ of open sets $U_i$ of  $S(\mathbb{C})^{an}$ such that $\overline{A}$ has non empty interior. 
\item[(iii)] The property $\mathcal P(s)$ holds for all complex vector fields $v \in S(\mathbb{C})$, outside a countable union of proper closed subvarieties defined over $\mathbb{Q}$.
\end{itemize}
 
This simple statement is already an interesting tool to study  the generic behavior of an $\emptyset$-definable family of vector fields as it allows to focus the analysis on specific open semi-algebraic subsets of the parameter space $S(\mathbb{C})$. In the case of complex planar vector fields, one can for example assume that the vector field admits an hyperbolic singularity --- namely that the eigenvalues of the linear part of $v$ along one of its zero are non zero and satisfy $\lambda/\mu \notin \R_-$ --- by restricting to appropriate open semi-algebraic subsets of $S(\mathbb{C})$.    \\

The methods of this article are flexible enough to obtain variants of Theorem \ref{introa} for other ``ample enough'' parameter spaces of planar algebraic vector fields. For example, the proof of Theorem \ref{introa} can be copied \textit{mutatis mutandis} to prove a variant of Theorem \ref{introa} where the degree of the vector field is replaced by the degree of the associated foliation (see \citep[Chapter V, Section 25]{Ily} for an extensive discussion concerning the differences between these two cases). 

A positive answer to Poizat's question for affine vector fields of higher dimensions --- more precisely, a generalization of Theorem \ref{introa} for vector fields of degree $d \geq 3$ on affine spaces $\mathbb{A}^n$ of arbitrary dimension --- is also expected, as well as for ample enough families of  vector fields over more general varieties than the affine space. I expect such generalizations to follow from similar methods to the ones of this article. This is the subject of ongoing work. \\

We now describe two consequences of Theorem \ref{introa} outside of model-theory: the first one concerns properties of integrability of the solutions of a general vector field of degree $d \geq 3$ while the second one involves algebraic independence of its solutions. \\

\textbf{Irreducibility of general planar vector fields.}  It is well known that the solutions of a general planar algebraic vector fields are transcendental functions (see \cite{Petro}). Among other things, Theorem \ref{introa} expresses a stronger non-integrability statement for the solutions of general planar vector fields of degree $d \geq 3$ known as \textit{irreducibility in the sense of Nishioka-Umemura} (see \citep[Section 2]{Umemura}). \\

A differential field extension $(K,\delta)$ of \textit{$1$-classical functions}  is a differential field with field of constants the field of complex numbers and that can  can be built from $(\mathbb{C}(t),  \frac d  {dt})$ as a tower 
$$(\mathbb{C}(t),  \frac d  {dt}) = (K_0,\delta_0) \subset (K_1,\delta_1) \subset \ldots \subset  (K_n,\delta_n) = (K,\delta_K)$$
   of differential fields  where each step of this tower is constructed either as an \textit{algebraic extension, a strongly normal extension or a differential extension of transcendence degree $1$}. Thanks to the first assumption, any field of $1$-classical functions can be realized as a differential field of meromorphic functions on some open set $U \subset \mathbb{C}$ (as described in Section 1 of \cite{Umemura}).  As noted in the appendix of \cite{Nag2}, these differential fields of $1$-classical functions are the ones associated to isolated types analyzable in the constants together with all order one differential equations (over arbitrary small sets of parameters).

The following irreducibility property for general planar algebraic vector fields of degree $d \geq 3$ is a direct consequence of Theorem \ref{introa}:

{\corx\label{cora} Let $v$ be a general planar algebraic vector field of degree $d \geq 3$ and let $(K,\delta)$ be any field of $1$-classical functions. Then:
$$(\mathbb{A}^2,v)^{(K,\delta)} = (\mathbb{A}^2,v)^{(\mathbb{C},0)} = \lbrace \text{constant solutions  of }  (\mathbb{A}^2,v)  \rbrace.$$}
 
 Here, $(\mathbb{A}^2,v)^{(K,\delta)}$ denotes the set of solutions of the differential equation $(\mathbb{A}^2,v)$ inside the differential field $(K,\delta)$. After writing  $v(x,y) = f(x,y) \frac \partial {\partial x} + g(x,y) \frac \partial {\partial y}$, it is the subset of $K^2$ of elements $(x,y)$ such that $\delta(x) = f(x,y)$ and $\delta(y) = g(x,y)$.
 
Corollary \ref{cora} expresses that no non constant solution of $(\mathbb{A}^2,v)$ belongs to a differential field of $1$-classical functions. Note that the set $(\mathbb{A}^2,v)^{(\mathbb{C},0)}$ of constant solutions of $(\mathbb{A}^2,v)$  is well-known: it is the finite set constituted by the zeros of $v$ or, in other words, the common zeros of the polynomials $f(x,y)$ and $g(x,y)$ (which has cardinality $d^2$ for a general planar algebraic vector field of degree $d$). \\

\textbf{Algebraic independence of the solutions of general planar vector fields.} The study of closed subvarieties of a complex algebraic variety $X$ invariant under the action of the analytic flow of a vector field $v$ is also an active theme of the theory of complex dynamical systems.  Algebraically, a complex  closed subvariety $Z$ of $X$ is called \textit{invariant} if the associated sheaf of ideals $\mathcal I_Z$ of $\mathcal O_X$ is invariant under the derivation $\delta_v$ induced by $v$ on $\mathcal O_X$. \\

Closed invariant subvarieties  of general vector fields on algebraic varieties have been extensively studied by geometers: the case of planar polynomial vector fields is first studied in \cite{Petro} (see \citep[Appendix, Chapter 5]{Ily} for a presentation in a modern language), while higher dimensional cases are studied in \cite{Lor} and \cite{Per} in the language of foliations. 

 An equally interesting question --- strongly motivated by the point of view of model theory on differential equations --- is to understand  more generally the algebraic relations among analytic solutions of a differential equation $(X,v)$ or, in other words, the closed subvarieties of $X^n$ invariant under the product vector field $v \boxplus \ldots \boxplus v$. In the language of model theory, the natural generalization of Landis-Petrovskii's result in \cite{Petro} takes the following form:

{\conj \label{remi} The structure of the solutions of a general planar algebraic vector field of degree $d \gg 0$ is an expansion of a strictly disintegrated set (a pure infinite set with no structure at all) by $d^2$ constant symbols to name the complex singularities of the vector field $v$.}
\vspace{0.2cm}

This conjecture would achieve a complete classification of the algebraic relations shared by the solutions of a differential equation associated to a general planar algebraic vector field.  Formulated at the level of analytic solutions, the previous conjecture can be split into a sequence $((C_n) : n \in \mathbb{N})$ of statements concerning $n$-tuples of analytic solutions of a general vector field $v$ of degree $d \gg 0$:  \\

\textit{
$(C_n): $ If $\gamma_1, \ldots , \gamma_n$ are $n$ distinct non stationary analytic solutions of the differential equation $(\mathbb{A}^2,v)$, then $\gamma_1, \ldots, \gamma_n$ are algebraically independent over $\mathbb{C}$.} \\

An analytic solution $\gamma$ of $(\mathbb{A}^2,v)$ is stationary if its image is a point and, since it is a solution of the vector field $v$, this point has to be one of the zeroes of $v$.  In that form, the statement $(C_1)$ is equivalent to Landis-Petrovskii Theorem while $(C_2)$  is an open question. A consequence of Theorem \ref{introa} is that in fact for $n \geq 2$, the implication  $(C_2) \Rightarrow (C_n)$ holds for general planar vector fields:

{\corx Let $d \geq 3$ and $v$ be a general planar algebraic vector field of degree $d$.  Then $(\mathbb{A}^2,v)$ satisfies the disintegration property.

In particular, for all $n \geq 2$, the implication $(C_2) \Rightarrow (C_n)$ holds for any general planar vector field of degree $d \geq 3$.}

Recall that we say that the vector field $v$ \textit{satisfies the disintegration property} if any irreducible closed invariant subvariety of $(X,v) ^n$ for some $n \geq 3$ can be written as an irreducible component of
$$\bigcap_{1 \leq i \leq j \leq n} \pi_{i,j}^{-1}(Z_{i,j})$$
where $\pi_{i,j}: X^n \rightarrow X^2$ denotes the projection on the coordinates $i$ and $j$ and $Z_{i,j}$ is an irreducible closed invariant subvariety of $(X,v) \times (X,v)$. \\

\textbf{Trichotomy in $\textbf{DCF}_0$.} The strategy for the proof of Theorem \ref{introa} follows the strategy of \cite{moi2} and \cite{moi3} to study geodesic flows of pseudo-Riemannian varieties with negative curvature of dimension two. \\

This strategy focuses on the study of the generic type of the differential equation $(X,v)$ and, more importantly, relies on the Trichotomy theorem of Hrushovski and Sokolovic and on \textit{the full classification of non-disintegrated locally modular types} of \cite{Sok} (more precisely, on the fact that all non-disintegrated locally modular types are orthogonal to any type with parameters in the constants, which is itself a consequence of the full classification). Using this powerful result, the proof of Theorem \ref{introa} reduces to the two following steps:
\begin{itemize}
\item[(a)] Establish that the generic type of $(X,v)$ is orthogonal to the constants.
\item[(b)] Establish that the differential equation $(X,v)$ does not admit non-trivial algebraic factors.
\end{itemize}

Recall that a \textit{rational factor} of a differential equation $(X,v)$ is a dominant rational morphism $\phi: (X,v) \dashrightarrow (Y,w)$ in the category of algebraic varieties endowed with vector fields, namely a dominant rational morphism $\phi: X \dashrightarrow Y$ satisfying $d\phi(v) = w$. An \textit{algebraic factor} of $(X,v)$ is a rational factor of a generically finite extension (or  a generically finite cover) of $(X,v)$. Following the terminology of \cite{Bui}, an algebraic variety $X$ endowed with a vector field $v$ will also be called a \textit{$D$-variety} $(X,v)$. \\

The property of orthogonality to the constants has been extensively studied for differential equations defined over constant parameters in \cite{moi}. In particular, (a) was settled in \cite{moi} for affine general algebraic vector fields of degree $d \geq 3$ in any dimension (see \citep[Théorème D]{moi}).  One of the main additional results of this article is the following geometric criterion preventing the existence of non-trivial algebraic factors of a vector field on a smooth complex algebraic variety of dimension two:
  
{\thmx\label{introc} Let $(X,v)$ be a smooth, irreducible complex $D$-variety of dimension $2$. Denote by $\mathcal F(v)$ the foliation on $X$ tangent to $v$. Assume there exists a zero $p \in X(\mathbb{C})$ of $v$ such that:
\begin{center}
(no algebraic separatrix) there is no complex closed algebraic curve on $X$ containing $p$ and invariant  under $v$.
 \end{center}
 Then the only algebraic webs on $X$ invariant under the vector field $v$ are of the form $$\mathcal W = \mathcal F(v) \boxtimes \cdots \boxtimes \mathcal F(v).$$
 Moreover, $(X,v)$ does not have any algebraic factor of dimension $1$.}

\vspace{0.3cm}
To deduce Theorem \ref{introa}, we will only use the second part of Theorem \ref{introc}, namely that under the hypotheses of Theorem \ref{introc}, $(X,v)$ does not admit any algebraic factor of dimension one.  The point of view of this article is to study these algebraic factors as part of a more general problem: the classification of  foliations and webs living on $X$ and \textit{invariant under the Lie derivative of the vector field $v$}.

We give a self contained introduction to algebraic foliations and webs on a smooth algebraic variety $X$ of dimension two in the first section. A central object is the projective bundle $\pi: \mathbb{P}(T_X) \rightarrow X$ formed by the lines of the tangent bundle $T_X$ of $X$. Foliations and webs on $X$ correspond respectively to the rational and  to the algebraic sections of this projective bundle.

We show in section 2.2 that the problem of classification of foliations and webs invariant under a vector field $v$  reduces to finding the algebraic solutions (in $\mathbb{C}(X)^{alg}$) of a Ricatti equation with parameters in the (differential) function field $\mathbb{C}(X)$ of $X$. Once such a classification has been established, one obtains the rational and algebraic factors of $(X,v)$ by studying respectively when these foliations and webs are \textit{algebraically integrable}. 

Theorem \ref{introc} says that --- for $v$ is a vector field on a smooth complex surface $X$ ---  if $v$ vanishes at a point and that this point is not contained in any closed algebraic curve $C$ on $X$ invariant under $v$, then there are as few webs as possible on $X$ invariant under the vector field $v$. The analysis of algebraic factors then follows as described above. \\

The article is organized as follows: in the first section, we gather some facts about foliations and webs on smooth algebraic surfaces in the setting of complex algebraic geometry. Classical references are given at the beginning of the section.

Differential algebra starts in the second section where we describe the main notion of the article: foliations and webs invariant under a vector field. We describe extensively this notion for a smooth complex $D$-variety $(X,v)$ of dimension two and we deduce Theorem \ref{introc} from our analysis.

The third section is the model-theoretic section where we combine Theorem \ref{introc}  with arguments from geometric stability theory to prove Theorem \ref{introa}. In the last section, we follow routine procedures to deduce the two corollaries of this introduction from Theorem \ref{introa}.  \\

%

\textbf{Acknowledgments. } I am very grateful to Jean-Beno\^it Bost for suggesting that foliation theory would be a useful tool to study problems (1) and (2). I am also grateful to J.V. Pereira for different suggestions during the meeting ``Algebraic geometry and foliations'' in June 2019 in Grenoble, which led to a substantial improvement of the first version of this article. Finally, I also would like to thank Levon Haykazyan, Martin Hils, Rahim Moosa and Anand Pillay for many interesting discussions concerning this article.

\tableofcontents

\section{Preliminaries on algebraic foliations and webs in dimension two}

In this section, we fix $X$ \textbf{a smooth complex algebraic variety of dimension two}. $T_X$ denotes the complex tangent space of $X$ and $\Omega^1_X$ the sheaf of one-forms on $X$. Since $X$ is assumed to be smooth, $\pi: T_X \rightarrow X$ is a vector bundle over $X$ and $\Omega^1_X$ is a locally free sheaf of rank two.

We also consider the projectivization $\pi: \mathbb{P}(T_X) \rightarrow X$ of the tangent bundle on $X$ whose complex points correspond to the set of complex lines in $T_X(\mathbb{C})$ and the associated diagram:

$$\xymatrix{T_X \setminus \lbrace 0-\text{section} \rbrace \ar[d] \ar[r]  &  \mathbb{P}(T_X) \ar[ld] \\ X}$$ 
 
Classical references for webs and foliations in dimension two are \cite{Bru} and \cite{Pirio} (see also section 2 of \cite{Favre-Pereira}).

\subsection{Algebraic foliations with singularities} A \textit{regular foliation} on $X$ is equivalently:
\begin{itemize}
\item a regular section of $\pi: \mathbb{P}(T_X) \rightarrow X$
\item a sub vector bundle of rank one $F$ of the vector bundle $T_X$. 
\end{itemize}

{\defn A (possibly singular) \textit{foliation} on $X$ is an invertible subsheaf $\mathcal F$ of $\Omega^1_X$ such that $\Omega^1_X/\mathcal F$ does not have torsion.}

So if $X = \mathrm{Spec}(A)$ is affine, a possibly singular foliation on $X$ is an $A$-submodule $M_\mathcal F$ of $\Omega^1_A$ which is free of rank one and such that $\Omega^1_A/M_A$ does not have torsion.  This is a weakening of the notion of regular foliation since in the regular case $\Omega^1_A/M_A$  is free of rank one (so in particular, torsion free).

{\Prop[{\citep[Proposition 2.3.1]{moi3}}]  \label{saturationoffoliation} Let $U \subset X$ be a dense open set. Any foliation on $U$ extends uniquely to a foliation on $X$.

In particular, there is natural correspondence between the set of (possibly singular) foliations on $X$ and the set of rational sections of $\pi: \mathbb{P}(T_X) \rightarrow X$.}

We will use the following decomposition of a foliation $\mathcal F$ on $X$ into its (finite) singular part and its regular part:
\begin{itemize}
\item a finite set $S = \mathrm{Sing}(\mathcal F) \subset X$ of \textit{singularities}, which are the points $x \in X$ such that  $\Omega_X/\mathcal F$ is not locally free at $x$.
\end{itemize}
If $\sigma_\mathcal F: X \dashrightarrow \mathbb{P}(T_X)$ denotes the rational section corresponding to $\mathcal F$ then  $\mathrm{Sing}(\mathcal F)$ is the set of points where this rational section is not regular.
\begin{itemize}
\item a regular foliation $F$ on the open set $U = X \setminus S$: every $p \in U$ admits a neighborhood $V$ such that $\mathcal F_{\mid V}$ is generated by a one-form $\omega$ which does not vanish on $V$ and
$$ F_{\mid V} =  \lbrace (x,v) \in T_V \text{ ,  } \omega_x(v) = 0 \rbrace$$
is a subvector bundle of $T_V$.
\end{itemize}

\rem If $p \in S$ is a singularity of $\mathcal F$, there still exists a neighborhood $V$ such that $\mathcal F_{\mid V}$ is generated by a one-form $\omega$  but the one-form $\omega$ vanishes at $p$. So, if $V^\ast = V \setminus \lbrace p \rbrace $, we can still write 
$$F_{\mid V^\ast} =  \lbrace (x,v) \in T_{V^\ast} \text{ ,  } \omega_x(v) = 0 \rbrace.$$
  
The study of singularities of foliations and more precisely of normal forms for the previous expression in both analytic and formal coordinates plays a central role in the birational study of foliations (see \citep[Chapter 1]{Bru}).

\subsection{Algebraic webs with singularities} Algebraic webs play the same role than foliations but for algebraic sections of $\pi: \mathbb{P}(T_X) \rightarrow X$ instead of rational ones.

\nota \label{notationweb} We denote by 
$$\mathrm{Sym}^\bullet (\Omega^1_X)  = \oplus_{r \in \mathbb{N}} \mathrm{Sym}^r (\Omega^1_X) $$
the sheaf of symmetric graded $\mathcal O_X$-algebra attached to the sheaf $\Omega^1_X$ of one-forms on $X$ and we denote the product in this algebra by:
$$\boxtimes : \mathrm{Sym}^r (\Omega^1_X)  \times \mathrm{Sym}^s (\Omega^1_X)  \rightarrow \mathrm{Sym}^{r + s} (\Omega^1_X) $$
 
 For every $r \in \mathbb{N}$, $\mathrm{Sym}^r (\Omega^1_X)$ is a locally free sheaf on $X$ of rank  $r + 1$ that can be described locally (in étale coordinates) as follows:  $X$ can be covered by affine open sets  $U$ with functions $x,y \in \mathcal O_X(U)$ such that  
 $$(x,y): U \rightarrow  \mathbb{A}^2 \text{ is étale}.$$
This means that  $ \Omega^1_X(U) =  \mathcal O_X(U) dx \oplus  \mathcal O_X(U) dy$ so that we get an isomorphism of graded $\mathcal O_X(U)$-algebras
$$ \mathrm{Sym}^\bullet (\Omega^1_X)(U) \cong \mathcal O_X(U)[dx,dy]$$
where the right-hand term is the polynomial algebra over the ring  $\mathcal O_X(U)$ with indeterminates $dx$ and $dy$, graded by the total degree.

{\lem Let $U \subset X$ be an affine open set endowed with étale coordinates $x,y \in \mathcal O_X(U)$ and denote by $p : T_X \rightarrow X$ the canonical projection. Then $p^{-1}(U)$ is affine and 
$$\mathcal O_{T_X}(\pi^{-1}(U)) \cong  \mathrm{Sym}^\bullet (\Omega^1_X)(U) \cong \mathcal O_X(U)[dx,dy] $$}

See Ex 5.16, Chapter II of  \cite{Hartshorne-book}. 

{\defn Let $r \geq 1$. A (possibly singular) \textit{$r$-web of foliations by curves on $X$}  is an invertible subsheaf $\mathcal W$ of the locally free sheaf $\mathrm{Sym}^r(\Omega^1_{X/k})$ such that  $\mathrm{Sym}^r(\Omega^1_{X/k})/\mathcal W$ does not have torsion. }

For $r = 1$, a $1$-web is identified with a (possibly singular) foliation on $X$ through the identification $\mathrm{Sym}^1(\Omega^1_X) = \Omega^1_X$. Generalizing the previous notion for foliations, we define $\mathrm{Sing}(\mathcal W)$ as the set of points $x \in X$ such that  $\mathrm{Sym}^r(\Omega^1_{X/k})/\mathcal W$ is not locally free at $x$.
Since  $\mathrm{Sym}^r(\Omega^1_{X/k})/\mathcal W$ does not have torsion, $\mathrm{Sing}(\mathcal W)$ is always a finite set.

{\Prop \label{web-saturation} Let $U \subset X$ be a dense open set and $r \geq 1$. Any $r$-web on $U$ extends uniquely to an $r$-web on $X$.}

\begin{proof}
Let  $\mathcal W$ be an $r$-web on $U$. There exists a coherent subsheaf $\mathcal W_X$ of $\mathrm{Sym}^r(\Omega^1_{X/k})$ such that $\mathcal W_{X |U} = \mathcal W$: namely, the sheaf $\mathcal W$  of $r$-symmetric $1$-forms such that $\omega_{|U} \in \mathcal W$, which, as a simple verification shows, is quasi-coherent, hence coherent (see Ex 5.15, Chapter II of  \cite{Hartshorne-book}). 

Let $\overline{\mathcal W}$ be the saturation of $\mathcal W_X$ in $\mathrm{Sym}^r(\Omega^1_{X/k})$(see \cite{Har}). By definition, $\overline{\mathcal W}$ is a saturated subsheaf of $\mathrm{Sym}^r(\Omega^1_{X/k})$, whose restriction to $U$ is $\mathcal W$. 
\end{proof}

{\defn A \textit{horizontal divisor $D$ of $\mathbb{P}(T_X)$} is a Weil-divisor $D = \sum k_i Z_i$ where $k_i \in \mathbb{Z}$ are integers and $Z_i$  are closed irreducible hypersurfaces of  $\mathbb{P}(T_X)$  which dominate $X$ that is: 
$$ \pi_i: \pi_{\mid Z_i}: Z_i \rightarrow X \text{ is dominant.}  $$ }

This last requirement is equivalent to $\pi_i$ surjective using the projectivity of $Z_i$ over $X$. Moreover, the generic fibre of $\pi_i$ is always finite. We denote $n_i$ the cardinality of the generic fibre and set:
$$\mathrm{deg}(D) = \sum k_in_i \in \mathbb{Z}.$$

By definition, the set  $\mathrm{Div}_h(\mathbb{P}(T_X))$ of horizontal divisors of $\mathbb{P}(T_X)$  is an abelian group and $\mathrm{deg}$ is   a morphism of groups. Recall that a divisor $D = \sum k_i Z_i$ is effective if $k_i \geq 0$ for every $i$.

 {\Prop \label{equivalence webs algebraic sections} Let $r \geq 1$. There is a one to one correspondence $\mathcal W \mapsto D(\mathcal W)$ between:
\begin{itemize}
\item[(i)] the set of  $r$-webs on $X$,
\item[(ii)] the set of horizontal divisors $D \in \mathrm{Div}_h(\mathbb{P}(T_X))$ which are effective and of degree $r$.
\end{itemize}}

\begin{proof}
If $U \subset X$ is a dense open set, by definition of horizontal divisors,  the restriction homomorphism  
$$\mathrm{Div}_h(\mathbb{P}(T_X)) \rightarrow  \mathrm{Div}_h(\mathbb{P}(T_U))$$ is an isomorphism. Therefore, it suffices to prove the statement for some dense open set $U  \subset X$. So we assume that $X$ is affine and endowed with two étale coordinates $x,y \in  \mathcal O_X(X) = A$. With the notation \ref{notationweb}, we write  $$A[dx,dy] \simeq  Sym^\bullet(\Omega^1_A) \simeq \mathcal O_{T_X}(T_X) $$ 

Let $\mathcal W$ be an $r$-web on $X$. Since $X$ is affine, the web $\mathcal W$ is generated by a global symmetric $r$-form $$\omega = f_0 (dx)^r + f_1 dx(dy)^{r - 1} + \cdots + f_r (dy)^{r}   \in Sym^r(\Omega^1_A).$$ We define $Z_\mathcal W$  to be the divisor of $T_X$ corresponding to the vanishing of the symmetric $r$-form  $\omega$ seen as a global function on $T_X$ and $D(\mathcal W)$ its image in $\mathbb{P}(T_X)$ i.e.
$$ Z_\mathcal W := (\omega = 0)\text{  and } D(\mathcal W)  = \rho(Z_\mathcal W)$$
where $\rho: T_X \setminus \lbrace 0-section \rbrace \rightarrow \mathbb{P}(T_X)$. Note that, as $\omega$ is an homogeneous polynomial,  $D(\mathcal W)$ is indeed a divisor of $\mathbb{P}(T_X)$.

We easily see that:

\begin{itemize}
\item The definition of $D(\mathcal W)$ does not depend on the chosen generating symmetric $r$-form $\omega \in \mathrm{Sym}^r(\Omega^1_A)$ (as two such generators differ by multiplication by an element in $A^\ast$).
\item Since  $\mathrm{Sym}^r(\Omega^1_A)/\mathcal W$ does not have torsion, the set of common zeros of $f_0,\ldots f_r$ is finite and  the divisor $D(\mathcal W)$ is horizontal.
\end{itemize}

Conversely, given an effective horizontal divisor $D$ of degree $r$, we can write $D :=(\omega = 0)$  for some homogeneous polynomial $\omega \in A[dx,dy]$ of degree $r$. Denote by $\mathcal W = \mathcal W(D)$ the $A$-submodule of $Sym^r(\Omega^1_A)$ generated by $\omega$. Equivalently $\mathcal W(D)$ is the $A$-submodule of $\mathrm{Sym}^r(\Omega^1_A)$ of symmetric $r$-forms  which vanishes on $D$.  

{\asser Since $D$ is a horizontal divisor, $\mathrm{Sym}^r(\Omega^1_A)/ \mathcal W$ is torsion-free.}

\vspace{0.1cm}

Indeed, consider $f.\omega_0 \in \mathrm{Sym}^r(\Omega^1_A)$ which is sent to $0$ in $\mathrm{Sym}^r(\Omega^1_A)/ \mathcal W$. This means that   $f.\omega_0$ vanishes on $D$. Denote by $C$ the curve in $X$ defined by $f = 0$. Since $D$ is horizontal, $D \setminus \pi^{-1}(C)$ is (scheme theoretically) Zariski-dense in $D$.  

As $\omega_0$ vanishes on $D \setminus \pi^{-1}(C)$, it also vanishes on the Zariski-closure $D$ so that $\omega_0$ is sent to $0$ in $\mathrm{Sym}^r(\Omega^1_A)/ \mathcal W$. It follows that $\mathrm{Sym}^r(\Omega^1_A)/ \mathcal W$ is torsion free.

Moreover, by construction, $D = D(\mathcal W(D))$ and $\mathcal W = \mathcal W (D(\mathcal W))$. 
\end{proof}
%
%
%

{\defn Let $\mathcal W$ be an $r$-web on $X$. We write $D(\mathcal W) = \sum k_i Z_i$ where $k_i > 0$ and $Z_i$ irreducible hypersurfaces. We denote by $\mathcal W_{red}$ the web on $X$ corresponding to the horizontal divisor $Z_1 + \ldots +  Z_n \in \mathrm{Div}_h(\mathbb{P}(T_X))$.}

 \subsection{Operations on webs} 
 \subsubsection{Sum of webs} If $\mathcal W_1$ and $\mathcal W_2$ are respectively an $r_1$-web and an $r_2$-web on $X$, we can form the $( r_1 + r_2)$-web $ \mathcal W_1 \boxtimes \mathcal W_2$ on $X$ whose local sections are of the form $\omega_1 \boxtimes \omega_2$ where $\omega_1$ and $\omega_2$ are local sections of $\mathcal W_1$ and $\mathcal W_2$ respectively.
 
At the level of horizontal divisors, this operation corresponds to the sum of divisors.

{\defn  Let $\mathcal W$ be an $r$-web on $X$. We say that $\mathcal W$ is \textit{completely decomposable} if $\mathcal W$ is of the form:
$$ \mathcal W = \mathcal F_1 \boxtimes \cdots  \boxtimes \mathcal F_r$$
 where $\mathcal F_1,\ldots, \mathcal F_r$ are foliations on $X$.}

 \subsubsection{Pull-backs of webs} Let $\phi: X' \dashrightarrow X $ be a generically finite dominant morphism between smooth irreducible algebraic surfaces. There is a pull-back operation:
 $$\mathcal W \mapsto \phi^\ast \mathcal W$$
 which associates to an $r$-web on $X$, an $r$-web on $X'$.
 
It is defined as follows: consider any dense open set $U'$ in $X'$ such that $\phi_{|U'}: U' \rightarrow  X$ is regular and étale. Since $\phi$ is regular, we can pull back symmetric forms from $X$ to $U'$ and we define $\phi^\ast \mathcal W_{|U'}$ to be the web on $U'$ generated by local sections of the form $\phi^\ast \omega$ where $\omega$ is a local section of $\mathcal W$.

By Proposition \ref{web-saturation}, the web $\phi^\ast \mathcal W_{|U'}$ extends uniquely to a web on $X$ denoted $\phi^\ast \mathcal W$. This construction does not depend on the choice of the open set $U'$.

 \begin{lem}
Let $\mathcal W$ be an $r$-web on $X$. There exists a generically finite cover $m : X' \dashrightarrow X$ such that
 $m^\ast \mathcal W$ is completely decomposable. 
 \end{lem}
 
 \begin{proof}
 We prove that statement by induction on $r \geq 1$. The case $r = 1$ is vacuous so assume the statement holds for some $r \geq 1$. Let $\mathcal W$ be an $(r + 1)$- web on $X$. Let $Z$ be an irreducible component of $D(\mathcal W)$ and denote by $\phi: Z \rightarrow X$ the restriction of the canonical projection of $\mathbb{P}(T_X)$ to $Z$. 
 
By construction, $\phi$ is generically finite and $\phi^\ast \mathcal W$ admits a tautological section corresponding to $id: Z \rightarrow Z$ so we can write $$\phi^\ast \mathcal W = \mathcal F \boxtimes \mathcal W_2$$
 where $\mathcal F$ is the foliation on $Z$ corresponding to this section.
 
 Up to restricting $Z$ to a dense open set (or up to a birational modification), we can assume that $Z$ is smooth. Using the induction hypothesis, we obtain a generically finite cover $\phi_2: Z' \dashrightarrow Z$ such that $\phi_2^\ast \mathcal W_2$ is completely decomposable. It follows that $(\phi_2 \circ \phi)^\ast \mathcal W = \phi_2^\ast \mathcal F \boxtimes \phi_2^\ast \mathcal W_2 $ is completely decomposable.
 \end{proof}
 
 \subsubsection{Quotient by a finite group} Assume that $G$ is a finite group acting regularly on an affine variety $X$ such that the quotient $Y = X/G$ is smooth. We denote by $\phi: X \rightarrow Y$ the canonical projection.
 
 {\lem \label{pushforward-finitegroups} Let $\mathcal W$ be an $r$-web on $X$. With the notation above, denote by $\mathcal W = \lbrace \mathcal  W_1, \ldots , \mathcal W_s \rbrace$ the orbit of $\mathcal W$ under $G$. There exists a web $\mathcal W_Y$ on $Y = X/G$ such that:
 $$ \phi^\ast \mathcal W_Y = \mathcal W_{s_1} \boxtimes \cdots \boxtimes \mathcal W_{s_l}$$
 where $s_1,\ldots, s_l \in [1,\ldots , s]$.}

 \begin{proof}
The web $\mathcal W$ is generated by a symmetric $r$-form $\omega$. We denote by  $\lbrace \omega_1,\ldots,\omega_N \rbrace$ the orbit of $\omega$ under the action of $G$.
 
 By construction, $\omega =   \omega_1 \boxtimes \ldots \boxtimes \omega_N$ is $G$-invariant and $Y$ is smooth so there exists a symmetric form $\omega_Y$ on $Y$ such that
 $$ \phi^\ast \omega_Y = \omega_1 \boxtimes \ldots \boxtimes \omega_N$$

Denote by $\mathcal W_Y$ the web generated by $\omega_Y$ then $$\phi^\ast \mathcal W_Y = \mathcal W(\omega_1) \boxtimes \cdots \boxtimes \mathcal W(\omega_N)$$

where $\mathcal W(\omega_i)$ is the web generated by $\omega_i$. This proves the lemma since  $\mathcal W(\omega_i) \in \lbrace \mathcal  W_1, \ldots , \mathcal W_s \rbrace$.
 \end{proof}

\subsection{Analytic leaves of a foliation} Let $\mathcal F$ be an algebraic foliation on $X$ with finite set $S \subset X$ of singularities. We denote by $X^{an}$ the (complex) analytification of $X$ and by $\mathcal F^{an}$ the analytification of $\mathcal F$.

{\defn Let $\gamma: \mathbb{D} \rightarrow X^{an}$ be an analytic curve on $X$ defined on some open disk $\mathbb{D} \subset \mathbb{C}$. We say that \textit{$\gamma$ is tangent to the foliation $\mathcal F$} if 
$$ \gamma^\ast \omega = 0 \text{ for every local section } \omega \in \mathcal F(U).$$}

The relation on $X \setminus S$ defined for $x, y \in X \setminus S$ by:
\begin{center}
$x \sim_\mathcal F y$ if there exists an analytic curve tangent to the foliation $\mathcal F$ going through $x$ and $y$.
\end{center}
is an equivalence relation and the equivalences classes are analytic Riemannian surfaces immersed in $X$ called \textit{the leaves of $\mathcal F$}.

On the other hand, if $s \in S$ is a singularity of $\mathcal F$, an analytic curve tangent to $\mathcal F$ going through $s$ is called \textit{an analytic separatrix of $s$}.   It is a theorem of Camacho and Sad \cite{Camacho-Sad} that each singular point of $\mathcal F$ admits always at least one separatrix but in that case uniqueness can fail as witnessed  by the foliation tangent to the planar vector field $x \frac \partial {\partial x} + y \frac \partial {\partial y}$: $0$ is a singularity and all lines through $0$ are  tangent to the foliation.

{\cons Let $f : X \dashrightarrow C$ be a rational dominant map from $X$ to a smooth algebraic curve $C$.

The morphism $f$ is generically smooth so there exists an open set $U \subset X$ such that$ f_{|U}: U \rightarrow C$ is well-defined and smooth. It follows that the kernel $T_{X/C}$ of the differential 
$$df: T_X \rightarrow f^\ast T_C.$$  
is a subvector bundle of rank one and defines a foliation on $U$. By Proposition \ref{saturationoffoliation}, this foliation can be uniquely extended into a  (possibly singular) foliation on $X$ denoted  $\Theta_{X/C}$ and called the \textit{foliation  tangent to the fibres of $f$}.}

{\lem \label{algebraicallyintegrable}Let $\mathcal F$ be a foliation on $X$ with finite set $S \subset X$ of singularities. The following are equivalent:
\begin{itemize}
\item[(i)] the foliation $\mathcal F$ has algebraic leaves: if $L$ is a leaf of $\mathcal F$, then the Zariski-closure of $L$ is an algebraic curve of $X$.
\item[(i')] There are no analytic curve on $X^{an}$ tangent to $\mathcal F$ and Zariski-dense in $X$.
\item[(ii)] $ \mathcal F = \Theta_{X/C}$ for some rational dominant morphim $f: X \dashrightarrow C$.
\item[(ii')] There exists an open set $U \subset X$ and a morphism $f : U \rightarrow C$ such that the leaves of $\mathcal F_{|U}$ are the connected components of the fibres of $f$.
\end{itemize}}

When these conditions are realized, we say that $\mathcal F$ is an \textit{algebraically integrable foliation}. 

\begin{proof}
The equivalences $(i) \Leftrightarrow (i)'$ and $(ii) \Leftrightarrow (ii)'$ follow from the definitions. The equivalence $(i) \Leftrightarrow (ii)$ is Jouanoulou's theorem: a foliation with infinitely many algebraic leaves admits a rational integral. 
\end{proof}

{\defn Let $\mathcal W$ be an $r$-web on $X$. We say that $\mathcal W$ is \textit{algebraically integrable} if for some generically finite $\phi: X' \dashrightarrow X$ such that 
$$\phi^\ast \mathcal W = \mathcal F_1 \boxtimes \ldots \boxtimes \mathcal F_r $$
is completely decomposable, the foliations $\mathcal F_1 , \ldots,\mathcal F_r$ are all algebraically integrable.}

In that case, for every generically finite $\phi: X' \dashrightarrow X$ such that 
$\phi^\ast \mathcal W = \mathcal F_1 \boxtimes \ldots \boxtimes \mathcal F_s $ is completely decomposable, the foliations $\mathcal F_1 , \ldots,\mathcal F_s$ are algebraically integrable.
\subsection{Transversality and locus of tangency for webs} We first define the notion for foliations:

{\defn Let $\mathcal F$ and $\mathcal G$ be two foliations on $X$. We say that \textit{$\mathcal F$ and $\mathcal G$ meet transversely at a point $x \in X$} if $\mathcal F$ and $\mathcal G$ are not singular at $x \in X$ and the lines $\mathcal F_x \subset T_{X,x}$ and $\mathcal G_x \subset T_{X,x}$ are transverse lines.}

We denote by $\mathrm{tang}(\mathcal F,\mathcal G)$ the set of points $x \in X$ where $\mathcal F$ and $\mathcal G$ do not meet transversely.

{\lem \label{tangency-foliations} Let $\mathcal F$ and $\mathcal G$ be two distinct foliations by curves on $X$.  Then $\mathrm{tang}(\mathcal F,\mathcal G)$ is a closed algebraic subset of $X$ of pure codimension one.}

In other words, either $\mathrm{tang}(\mathcal F,\mathcal G) = \emptyset$ or $\mathrm{tang}(\mathcal F,\mathcal G)$ is a (possibly reducible) algebraic curve on $X$.  
\begin{proof}
It suffices to prove the statement locally on a covering of $X$ so we can assume that $X$ is affine endowed with an étale map 
$$(x,y): X \rightarrow \mathbb{A}^2.$$

This implies that $\Omega^1_X = \mathcal O_X dx \oplus \mathcal O_X dy.$  
Moreover, the foliations $\mathcal F$ and $\mathcal G$ are respectively generated by global one-forms $\omega_\mathcal F, \omega_\mathcal G \in \Omega^1_X(X)$ and we write

$$\omega_\mathcal F = f_1 dx + f_2dy \text{ and } \omega_\mathcal G = g_1 dx + g_2dy  $$ 
where $f_1,f_2,g_1,g_2 \in \mathcal O_X(X)$.

{\asser The following are equivalent:
\begin{itemize}
\item[(i)] $x \in \mathrm{tang}(\mathcal F,\mathcal G)$ .
\item[(ii)] The one forms $\omega_\mathcal F(x)$ and $\omega_\mathcal G(x)$ do not form a basis of the fibre of the cotangent space $T_{X,x}^\ast$ of $X$ over $x$.
\item[(iii)] the two-form $\omega_\mathcal F \wedge \omega_\mathcal G$ vanishes at $x$.
\item[(iv)] Let $D \in \mathcal O_X(X)$ be the determinant $D = f_1g_2-f_2g_1$ then $D(x) = 0$
\end{itemize}}

Since $\mathcal F$ and $\mathcal G$ are distinct, $D$ does not vanish identically on $X$. This implies that  $\mathrm{tang}(\mathcal F,\mathcal G)$ is a closed algebraic subset of $X$ of pure codimension one
\end{proof}

{\defn Let $\mathcal W_1$ and $\mathcal W_2$ be two webs on $X$. We write $W_1 = |D(\mathcal W_1)|$ and $W_2 = |D(\mathcal W_2)|$ which are two hypersurfaces of $\mathbb{P}(T_X)$. We say that $\mathcal W_1$ and $\mathcal W_2$ \textit{meet transversely} at $x \in X$ if
$$ W_{1,x} \cap W_{2,x} = \emptyset \subset \mathbb{P}(T_{X,x})$$
where $W_{i,x}$ denote the fibre over $x$ of the restriction $\pi_{|W_i}: W_i \rightarrow X$ of the canonical projection $\pi: \mathbb{P}(T_X) \rightarrow X$.}

We denote by $\mathrm{tang}(\mathcal W_1,\mathcal W_2)$ the set of points $x \in X$ where $\mathcal W_1$ and $\mathcal W_2$ do not meet transversely. Note that $\mathrm{Sing}(\mathcal W_1) \cup \mathrm{Sing}(\mathcal W_2) \subset  \mathrm{tang}(\mathcal W_1,\mathcal W_2)$ since $x$ is a singular point of $\mathcal W_i$ if and only if $W_{i,x} = \mathbb{P}(T_{X,x})$.

{\Prop \label{tangency-webs}  Let $\mathcal W_1$ and $\mathcal W_2$ be two webs on $X$.  Then
\begin{itemize}
\item[(i)] Either $\mathcal W_1$ and $\mathcal W_2$ share a common sub web: there exists a  web $\mathcal W_3$ on $X$ such that
$$W_3 \subset W_1 \cap W_2 \text{ where } W_i = |D(\mathcal W_i)| \subset \mathbb{P}(T_X) \text{ for } i =1,2,3.$$
\item[(ii)] or $\mathrm{tang}(\mathcal W_1,\mathcal W_2)$ is a closed algebraic subset of $X$ of pure codimension one.
\end{itemize}}
\begin{proof}
The statement is local so we can assume that $X$ is affine so that $\mathcal W_1$ and $\mathcal W_2$ are generated by global sections $\omega_1 \in \mathrm{Sym}^{r_1}(\Omega^1_X)(X)$ and $\omega_2 \in \mathrm{Sym}^{r_2}(\Omega^1_X)(X)$. We denote by $Res(\omega_1,\omega_2) \in \mathcal O_X(X)$ the determinant the $\mathcal O_X(X)$-linear map:

$$ \begin{cases}
\mathrm{Sym}^{r_2 - 1}(\Omega^1_X)(X) \times  \mathrm{Sym}^{r_1 - 1}(\Omega^1_X)(X) \rightarrow  \mathrm{Sym}^{r_1 + r_2 - 1}(\Omega^1_X)(X) \\ 
(\eta_1, \eta_2) \rightarrow \eta_1 \boxtimes \omega_1 + \eta_2 \boxtimes \omega_2 
\end{cases}$$

Note that $Res(\omega_1,\omega_2)$ is simply the resultant of two homogeneous polynomials of two variables under the identifcation of graded algebras $Sym^\bullet \Omega^1_X(X) \simeq \mathcal O_X(X)[dx,dy]$. It follows from the theory of the resultant that:

{\asser Let $x \in X$. The following are equivalent:
\begin{itemize}
\item[(i)] $ W_{1,x} \cap W_{2,x} \neq \emptyset \subset \mathbb{P}(T_{X,x})$
\item[(ii)] $Res(\omega_1,\omega_2)(x) = 0$
\end{itemize}}
 
 To conclude the proof, it remains to see that  if $Res(\omega_1,\omega_2) = 0$ at the generic point of $X$ then  $\mathcal W_1$ and $\mathcal W_2$  share a common sub web: indeed, since $W_1 \cap W_2$ projects generically on $X$ so does one of its irreducible component $Z$.
 
Since $Z$ dominates $X$, $Z$ has codimension one in $\mathbb{P}(T_X)$ and defines a horizontal divisor in $Div_h(\mathbb{P}(T_X)$. By Proposition \ref{equivalence webs algebraic sections}, $Z$ is of the form $|D(\mathcal W_3)|$ for some algebraic web $\mathcal W_3$ on $X$.
\end{proof}

\subsection{Discriminant of a reduced web}

{\defn Let $\mathcal W$ be an $r$-web on $X$. Denote by $W = |D(\mathcal W)| \subset \mathbb{P}(T_X)$. We say that the $r$-web $\mathcal W$ is \textit{smooth at x} if $|W_x| = r$ where $W_x$ is the fibre of $W$ over $x$.}

Note that $\pi_{|W}: W \rightarrow X$ is generically finite and that all finite fibres have cardinal $\leq r$ so that the set of points $x \in X$ such that $\mathcal W$ is smooth at $x$ form a Zariski open set $U \subset X$.

{\lem \label{generically-smooth-analytic} Let $\mathcal W$  be  an $r$-web on $X$.
\begin{itemize}
\item[(i)] The web $\mathcal W$ is generically smooth if and only if it is reduced.
\item[(ii)] The web $\mathcal W$ is smooth at $x \in X$ if and only if there exists an analytic neighborhood $U$ of $x \in X^{an}$ and analytic foliations $\mathcal F_1,\ldots, \mathcal F_r$ on $U$ pairwise transverse such that
$$\mathcal W_{\mid U}^{an} = \mathcal F_1 \boxtimes \ldots \boxtimes \mathcal F_r.$$
\end{itemize}}

\begin{proof}
\item[(i)] Clearly, if $\mathcal W$ is not reduced then it is not generically smooth. Conversely, assume that $\mathcal W$ is reduced.  We write $D(\mathcal W) = \sum Z_i$ where $Z_i$ are horizontal irreducible hypersurfaces.

Using Proposition \ref{equivalence webs algebraic sections}, we can write $Z_i = D(\mathcal W_i)$ for some web $\mathcal W_i$ on $X$. Since $Z_i$ and $Z_j$ are irreducible and distinct, $\mathcal W_i$ and $\mathcal W_j$ don't share a common sub web. By Proposition \ref{tangency-webs}, it follows that for $i \neq j$, $Z_{i,x} \cap Z_{j,x} = \emptyset$ for generic points $x \in X$. We conclude that
$$|W_x| = |\bigcup_{i} Z_{i,x}| = \sum_i |Z_{i,x}| = r \text{ for generic } x \in X.$$
has cardinal $r$.

\item[(ii)] We can assume that $\mathcal W$ is reduced (otherwise both of the statements fail for all $x \in X$). We denote by $U$ the (dense) open set of smooth points of $\mathcal W$ and consider
$$\pi_U: W \cap \mathbb{P}(T_U) \rightarrow U $$
Note that $\pi_U$ is proper and that all the fibres have the same cardinal $r$. It follows that $W \cap \mathbb{P}(T_U)$ is smooth and $\pi_U$ is an étale covering of degree $r$ of $U$. (ii) therefore follows from the fact that any étale covering of degree $r$ is locally trivial in the analytic topology.
\end{proof}

{\defn Let $\mathcal W$ be a reduced $r$-web on $X$. The closed subset of $X$ constituted of the points $x \in X$ such that $\mathcal W$ is not smooth at $x$  is called \textit{the discriminant locus of $\mathcal W$}  and denoted $\Delta(\mathcal W)$.} 

{\cor \label{discriminant of a web} Let $\mathcal W$ be a reduced $r$-web on $X$.  The discriminant locus $\Delta(\mathcal W)$ of $\mathcal W$ has pure codimension one.}
\begin{proof}
The statement is local so we can assume that $\mathcal W$ is generated by a global section $\omega \in \mathrm{Sym}^r(\Omega^1_X)(X)$.

We also fix an identification  $\mathrm{Sym}^\bullet(\Omega^1_X)(X) \simeq \mathcal O_X(X)[dx,dy]$ and consider the derivation induced on $\mathrm{Sym}^\bullet(\Omega^1_X)(X)$ by the  homogeneous derivation of degree $-1$ defined by

$$ \delta_x(dy) = 0 \text{, } \delta_x(dx) = 1  \text{ and } \delta_x{|\mathcal O_X(X)} = 0.$$ 
{\asser Denote by $\delta_x W$ the  $(r - 1)$-web generated by $\delta_x(\omega)$. Then 
$$ \Delta(W) = tang(\mathcal W,\delta_x \mathcal W)$$}

\begin{proof}
Indeed a homogeneous polynomial $P(x,y)$ does not have a multiple linear factor if and only if $Res(P,\delta_x(P)) = 0.$
\end{proof}

Since $\mathcal W$ is reduced, $\mathcal W$ and $\delta_x \mathcal W$ do not share a common subweb. It follows from Proposition \ref{tangency-webs} that $tang(\mathcal W,\delta_x \mathcal W)$ has pure codimension one.
\end{proof}

\section{Algebraic factors of a two-dimensional complex $D$-variety}

\subsection{Rational and algebraic factors} We work in the category of (smooth) irreducible $D$-varieties over the differential field $(\mathbb{C},0)$. The objects of this category are pairs $(X,v)$ where $X$ is a (smooth) complex algebraic variety endowed with a complex algebraic vector field $v$.

{\defn Let $(X,v)$ be a smooth and irreducible complex $D$-variety.
\begin{itemize}
\item A \textit{rational factor of $(X,v)$}, denoted $\phi: (X,v) \dashrightarrow (Y,w)$,  is a dominant rational morphism $\phi: X \dashrightarrow Y$ which satisfies $d\phi(v) = w$ at the generic point of $Y$ (or, equivalently, everywhere where $\phi$ is defined).

\item A \textit{generically finite extension} of $(X,v)$ is a rational dominant morphism $\phi: (X',v') \dashrightarrow (X,v)$ in the category of  complex $D$-varieties such that the generic fibre of $\phi$ is finite.

\item An \textit{algebraic factor of $(X,v)$} is a rational factor of some generically finite extension $(X',v')$ of $(X,v)$.
\end{itemize}}

The aim of this section is to prove the following theorem:

{\Thm\label{maintheorem} Let $(X,v)$ be a smooth, irreducible complex $D$-variety of dimension $2$. Denote by $\mathcal F(v)$ the foliation on $X$ tangent to $v$. Assume there exists a zero $p \in X(\mathbb{C})$ of $v$ such that:
\begin{center}
(no algebraic separatrix) there is no complex closed algebraic curve on $X$ containing $p$ and invariant  under $v$.
 \end{center}
 Then the only algebraic webs on $X$ invariant under the vector field $v$ are of the form $$\mathcal W = \mathcal F(v) \boxtimes \cdots \boxtimes \mathcal F(v).$$
 Moreover, $(X,v)$ does not have any algebraic factor of dimension $1$.}

\subsection{Invariant webs and foliations} Let $X$ be a smooth complex algebraic variety of dimension two  and $v$ be a vector field on $X$. Recall $v$ can be identified with a derivation $\delta_v$ on $\mathcal O_X$ defined by

$$ \delta_v(f) = df(v) \text{ for every local section } f \text{ of }   \mathcal O_X.$$

{\lem \label{lemma-Liederivative} Let $v$ be a vector field on $X$. There is a unique extension of the derivation $\delta_v$ on $\mathcal O_X$ into a derivation  $\mathcal L_v$ of the graded $\mathcal O_X$-algebra $ \mathrm{Sym}^\bullet(\Omega_X)$ satisfying:  
\begin{itemize}
\item $\mathcal L_v$ is an homogeneous derivation of degree $0$: $\mathcal L_v( \mathrm{Sym}^r(\Omega_X)) \subset \mathrm{Sym}^r(\Omega_X)$,
\item for every local function $f \in \mathcal O_X(U)$,  $\mathcal L_v(df) = d(\delta_v(f))$.
\end{itemize}}

\begin{proof}
We can assume that $X = \mathrm{Spec}(A)$ is affine. It is well known that any $D$-module structure on  $\Omega^1_A$  extends uniquely to a   homogeneous derivation of degree $0$ of  $Sym^\bullet (\Omega^1_A)$ so it suffices to show the existence of $\mathcal L_v$ on $\Omega^1_A$.
This follows from \citep[Lemma 1]{Ros} or \cite[Section 1.1]{moi3}.
\end{proof}
The derivation $\mathcal L_v$ is called the \textit{Lie-derivative of the vector field $v$} in differential geometry. Under the identification $p^\ast Sym^{\bullet}(\Omega^1_X) \simeq \mathcal O_{T_X}$ described in notation \ref{notationweb},  the derivation $\mathcal L_v$ defines a vector field denoted $Tv$ on $T_X$.

{\Prop \label{differential stucture} Let $v$ be a vector field on $X$. The vector field $Tv$ descends to a vector field $\mathbb{P}(v)$ on $\mathbb{P}(T_X)$ such that the following diagram holds in the category of $D$-varieties:

$$\xymatrix{(T_X \setminus \lbrace 0-\text{section}\rbrace,Tv) \ar[d] \ar[r]  &  (\mathbb{P}(T_X),\mathbb{P}(v)) \ar[ld] \\ (X,v)}.$$}

 \begin{proof}
 By considering a covering of $X$, we can assume that $X$ is affine and endowed with étale coordinates $(x,y)$ and we write
 $v = a \frac \partial {\partial x} + b \frac \partial {\partial y}$ with $a,b \in \mathcal O_X(X)$.
 
 Note that if $T_X$ is also affine with étale coordinates $(x,y,dx = \xi_1, dy= \xi_2)$ and that $\mathbb{P}(T_X)$ is covered by two affine open sets
$$U_1 =: (dx \neq 0) \text{ and } U_2:= (dy \neq 0)$$
 with respective étale coordinates $(x,y,\xi_1/\xi_2)$ and $(x,y,\xi_2/\xi_1)$. Using Lemma \ref{lemma-Liederivative}, we get that
  $$\mathcal L_v(\xi_1) = \mathcal L_v(dx) = \frac {\partial a} {\partial x} dx + \frac {\partial a} {\partial y} dy \text{  and  } \mathcal L_v(\xi_2) =  \frac {\partial b} {\partial x} dx + \frac {\partial b} {\partial y} dy.$$ 
 
 Computing $\mathcal L_v(\xi_1/\xi_2)$ using the Leibniz rule, we deduce that  
 $\xi_1/\xi_2$ satisfies the Ricatti equation with parameters in $\mathcal O_X(X)$:

$$ \mathcal L_v(\xi_1/\xi_2) = (\frac {\partial a} {\partial x} - \frac {\partial b} {\partial y} ) \xi_1/\xi_2  + \frac {\partial a} {\partial y} - \frac {\partial b} {\partial x}  (\xi_1/\xi_2)^2$$
 
 So the vector field $\mathbb{P}(v)$ on $U_1$ in the coordinates $(x,y,t_1 = \xi_1/\xi_2)$ is given by:
 $$\mathbb{P}(v)(x,y,t_1) = a \frac \partial {\partial x} + b \frac \partial {\partial x} +  \Big[ ( \frac {\partial a} {\partial x} - \frac {\partial b} {\partial y} )t_1 + \frac {\partial a} {\partial y} - \frac {\partial b} {\partial x}  t_1^2 \Big] \frac \partial {\partial t_2}$$

 A similar computation for $t_2 = \xi_2/\xi_1$ show that the in the open set $U_2$ the vector field $\mathbb{P}(v)$ is given by: 
 
 $$\mathbb{P}(v)(x,y,t_2) = a \frac \partial {\partial x} + b \frac \partial {\partial x} +  \Big[ ( \frac {\partial b} {\partial y} - \frac {\partial a} {\partial x} )t_2 + \frac {\partial b} {\partial x} - \frac {\partial a} {\partial y}  t_2^2 \Big] \frac \partial {\partial t_1}.$$
 
 Since $U_1$ and $U_2$ cover $\mathbb{P}(T_X)$ and the two vector fields agree on $U_1 \cap U_2$, we obtain a global vector field on $\mathbb{P}(v)$ satisfying the conclusion of the proposition.
 \end{proof}
%

{\Prop \label{equivalence-invariantweb} Let $v$ be a vector field on $X$,  let $\mathcal W$ be an $r$-web on $X$ and denote by $W$ the associated horizontal divisor of $\mathbb{P}(T_X)$. TFAE:
\begin{itemize}
\item[(i)] $\mathcal W$ is stable under the Lie-derivative of $v$: $ \mathcal L_v(\mathcal W) \subset \mathcal W$
\item[(ii)] The divisor $W$ is a closed subscheme of $\mathbb{P}(T_X)$ invariant under $\mathbb{P}(v)$: the (principal) ideal defining $W$ in $\mathbb{P}(T_X)$ is invariant under the derivation induced by $\mathbb{P}(v)$ on $\mathbb{P}(T_X)$. 
\item[(iii)] All the irreducible components of $W$ are invariant under $\mathbb{P}(v)$.
\end{itemize}}

\begin{proof}
$(i) \Rightarrow (ii)$ is obvious and $(ii) \Rightarrow (iii)$ is well-known (see Appendix of \cite{moi}). To prove that $(iii) \Rightarrow (i)$, we first assume that $W$ is irreducible. 

We denote the canonical projections by $\rho: T_X \setminus \lbrace 0-\text{section} \rbrace \rightarrow \mathbb{P}(T_X)$ and $ p: T_X \rightarrow X$. By Proposition \ref{differential stucture}, $\overline{\rho^{-1}(W)}$ is a closed invariant subvariety of $(T_X,Tv)$.

Let $U$ an affine open set with étale coordinates $x,y \in \mathcal O_X(U)$ and $\omega \in \mathrm{Sym}^r(\Omega^1_X(X))$ a symmetric $r$-form generating the ideal defining $Z = \overline{\rho^{-1}(W)}$ . Then $\mathcal W(U)$ is generated by $\omega$ (see the proof of Proposition \ref{equivalence webs algebraic sections}). Since $Z$ is invariant, there exists $h \in \mathcal O_{T_X}(p^{-1}(U))$ such that:
$$\mathcal L_v(\omega) = h\omega.$$

But $\mathcal L_v$ is homogeneous of degree $0$, so $\mathcal L_v(\omega)$ is a symmetric $r$-form and $h$ has degree $0$. It follows easily that $\mathcal W_{|U} = \mathcal O_U.\omega$ is invariant and after covering $X$ by such open sets that $\mathcal W$ is invariant.

The non-irreducible case follows from the fact that if $\mathcal W = \mathcal W_1 \boxtimes \mathcal W_2$ and $\mathcal W_1,\mathcal W_2$ are invariant then $\mathcal W$  is invariant (straightforward using the Leibniz rule).
\end{proof}

{\defn Let $v$ be a vector field on $X$ and $\mathcal W$ be an $r$-web on $X$. We say that the web $\mathcal W$ is \textit{invariant under the vector field $v$} if the equivalent properties of Lemma \ref{equivalence-invariantweb} hold.}

{\cor \label{invariant-reduced} Let $v$ be a vector field on $X$ and $\mathcal W$ be an invariant web. Then $\mathcal W_{red}$ is invariant.}

\subsection{Analytic interpretation of invariance} We now give an analytic interpretation of the previous definition for reduced webs.

{\nota Given a vector field $v$ on $X$ and $\epsilon > 0$, we call an analytic open set $U \subset X^{an}$ \textit{$\epsilon$-complete} (relatively to $v$) if for every $x \in U$, the unique analytic solution $\gamma(x)$ of $v$ through $x$ is defined at all times $t \in \mathbb{C}$ of norm less than $\epsilon$.

If $U$ is $\epsilon$-complete and $|t| \leq \epsilon$, we denote by $\phi_t : U \rightarrow X$ the local analytic flow of $v$ which associates to a point $x$ the value $\gamma_x(t)$ at time $t$ of the unique solution $\gamma_x$ of $v$ through $x$.}

{\rem Let $p: T_X \rightarrow X$ be the canonical projection and let $v$ be a vector field on $X$. If $U \subset X^{an}$ is an $\epsilon$-complete analytic open set then $p^{-1}(U) = T_U \subset T_X^{an}$ is an $\epsilon$-complete analytic open set for the vector field $Tv$ on $T_X$ and the local analytic flow  of $Tv$ on $p^{-1}(U)$ is the differential $d\phi_t: T_U \rightarrow T_X^{an}$ of the local analytic flow $\phi_t: U \rightarrow X^{an}$ of $v$ on $U$.

Moreover, $\phi_t$ being a local analytic isomorphism,  the differential $d\phi_t$ is invertible and sends lines to lines of $T_X$. This defines a local analytic flow 
$$\mathbb{P}(d\phi_t): \pi^{-1}(U) \rightarrow \mathbb{P}(T_X)$$  
where $\pi: \mathbb{P}(T_X) \rightarrow X$ is the canonical projection. So $\pi^{-1}(U)$ is $\epsilon$-complete for the vector field $\mathbb{P}(v)$ too.}

{\Prop\label{invariance for webs} Let $v$ be a vector field on $X$ and let $\mathcal W$ be a reduced  $r$-web on $X$.  Denote by $W = |D(\mathcal W)| \subset \mathbb{P}(T_X)$ The following are equivalent:
\begin{itemize}
\item[(i)] The web $\mathcal W$ is invariant under the vector field $v$.
\item[(ii)] For all $\epsilon > 0$ and all $\epsilon$-complete open set $U \subset X^{an}$, $$\mathbb{P}(d\phi_t)(W(\mathbb{C}) \cap \mathbb{P}(T_U)) \subset W(\mathbb{C})$$. 
\item[(iii)] For all $\epsilon > 0$ and all $\epsilon$-complete open set $U \subset X^{an}$, we have $$\phi_t^\ast \mathcal W (U) \subset \mathcal W(\phi_t^{-1}(U)) \text{ for all } t \text{ with } |t| \leq \epsilon.$$
\item[(iv)] For some $\epsilon > 0$ and some $\epsilon$-complete non empty open set $U \subset X^{an}$, we have $$\phi_t^\ast \mathcal W (U) \subset \mathcal W(\phi_t^{-1}(U)) \text{ for all } t \text{ with } |t| \leq \epsilon.$$
\end{itemize}}

\begin{proof}
It suffices to prove (ii) and (iii) for $\epsilon$-complete open set $U \subset X^{an}$ which are contained in affine open sets with étale coordinates (as any $\epsilon$-complete open set will be covered by such open sets). So we can assume that $X = \mathrm{Spec}(A)$ is affine and that we are given two étale coordinates $(x,y) : X \rightarrow \mathbb{A}^2$  and we write as usual 

$$\Omega^1_X(X) = A dx + A dy \text{ and } \mathrm{Sym}^\bullet(\Omega^1_X(X)) = A[dx,dy].$$

Note that $T_X = \mathrm{Spec}(A[dx,dy])$ so that, as before, every symmetric $r$-form on $X$ can be identified with a function on $T_X$. Since $X$ is affine, the web $\mathcal W$ is generated by a global section $\omega \in \mathrm{Sym}^r(\Omega^1_X(X))$

$(i) \Rightarrow (ii)$ We denote by $p: T_X \setminus \lbrace 0-\text{section} \rbrace \rightarrow \mathbb{P}(T_X)$ and set $Z = \overline{\pi^{-1}(W)} \subset T_X$ which is a closed hypersurface of $T_X$. By definition of $W$, the ideal $I_Z$ defining $Z$ is generated by $\omega$.

Since $\mathcal L_v$ is the derivation associated to the vector field $Tv$ and $\mathcal W$ is invariant under the vector field $v$, the ideal defining $Z$ is invariant under the $\mathcal L_v$. It follows that $Z$ is a closed invariant subvariety of $(T_X,Tv)$. By \citep[Proposition 3.1.20] {moi} and the remark above, it follows that: 
$$d\phi_t(Z \cap T_U) \subset Z.$$

The corresponding statement for $W$ follows.

$(ii) \Rightarrow (iii)$ Since the web $\mathcal W$ is reduced, for every analytic open set $U \subset X^{an}$, we have:
$$\omega \in \mathcal W(U)^{an} \text{ if and only if } \omega \text{ vanishes identically on } W \cap T_U$$

So that $d\phi_t(W \cap U) \subset W$ implies that $\phi_t^\ast \mathcal W (U) \subset \mathcal W(\phi_t^{-1}(U))$.

$(iii) \Rightarrow (iv)$ For every vector field $v$, $\epsilon$-complete open sets $U \subset X^{an}$ do exist.

$(iv) \Rightarrow (i)$ Consider the global section $\mathcal L_v(\omega) \wedge \omega$ of the vector bundle $\Lambda^2(\mathrm{Sym}^r(\Omega^1_X))$ on $X$. 

On an analytic open set $U \subset X^{an}$ satisfying (iv), we have that
$$\mathcal L_v(\omega) \wedge \omega = (\frac d {dt}_{|t = 0}  \phi_t^\ast \omega) \wedge \omega =  \frac d {dt}_{|t = 0}  ( \phi_t^\ast \omega \wedge \omega ) = 0  $$

So the global section $\mathcal L_v(\omega) \wedge \omega$ vanishes identically on a non empty analytic set. Since $X$ is irreducible, $\mathcal L_v(\omega) \wedge \omega = 0 $ on $X$. It follows that $\mathcal L_v(\omega) = h \omega$ for some rational function $h$ on $X$ and that $\mathcal W$ is an invariant web.
\end{proof}

\subsection{Characteristic loci of invariant webs and foliations}

{\cor \label{invariance-singularlocus} Let $v$ be a vector field on $X$ and $\mathcal W $ be an $r$-web on $X$ invariant under the vector field $v$.

The singular locus of the web $\mathcal W$ is a finite subset of $X$ contained in the set of zeros of $v$.}

\begin{proof}
We already know that the singular locus of the web $\mathcal W$  is finite. So it suffices to prove that $\mathrm{Sing}(\mathcal W)$ is invariant (as the invariant points in $X(\mathbb{C})$ are precisely the zeros of $v$).

So for the sake of a contradiction assume that there exists an $\epsilon$-complete open set $U$, $p \in \mathrm{Sing}(\mathcal W) \cap U$ and $q =  \phi_t(p) \notin \mathrm{Sing}(\mathcal W)$ where $\phi_t$ denotes the local flow of the vector field $v$.

Since $q \notin \mathrm{Sing}(\mathcal W)$, there exists a local section $\omega \in \mathcal W^{an}_q$ around $q$ which does not vanish at $q$. By Proposition \ref{invariance for webs},  $\phi_t^\ast \omega \in  \mathcal W^{an}_p$ and does not vanish at $p$ since $\phi_t$ is local isomorphism. It follows that $p \notin \mathrm{Sing}(\mathcal W)$ which is a contradiction.
\end{proof}

{\cor \label{discriminant-invariant} Let $v$ be a vector field on $X$ and $\mathcal W $ be a reduced $r$-web on $X$ invariant under the vector field $v$.

The discriminant $\Delta(\mathcal W)$ is either $\emptyset$ or a closed algebraic curve of $X$ invariant under $v$.}

\begin{proof}
Using Corollary \ref{discriminant of a web}, it is enough to prove that $\Delta(\mathcal W)$  is invariant. So for the sake of a contradiction assume that there exists an $\epsilon$-complete open set $U$, $p \in \Delta(\mathcal W) \cap U$ and $q =  \phi_t(p) \notin \Delta(\mathcal W)$ where $\phi_t$ denotes the local flow of the vector field $v$.

Since $q \notin \Delta(\mathcal W)$, using Lemma \ref{generically-smooth-analytic}, there are local analytic pairwise transverse
foliations $\mathcal F_1, \ldots, \mathcal F_r$ pairwise transverse such that
$$  \mathcal W_q = \mathcal F_1 \boxtimes \ldots \boxtimes \mathcal F_r$$ 

By Proposition \ref{invariance for webs}, we deduce that:
$$\mathcal  W_p = \phi_t^\ast W_q = \phi_t^\ast \mathcal F_1 \boxtimes \ldots \boxtimes  \phi_t^\ast \mathcal F_r.$$

Since $\phi_t$ is a local isomorphism, the foliations  $\phi_t^\ast \mathcal F_1, \ldots,   \phi_t^\ast \mathcal F_r$ are pairwise transverse so $p \notin \Delta(\mathcal W)$. Contradiction.
\end{proof}

{\cor \label{invariance-tangency} Let $v$ be a vector field , $\mathcal W_1$ and $\mathcal W_2$ be webs on $X$  invariant under $v$. Then $tang(\mathcal W_1,\mathcal W_2)$ is a  closed invariant subvariety of $(X,v)$.}

\begin{proof}
We can assume that $\mathcal W_1$ and $\mathcal W_2$ are reduced by Corollary \ref{invariant-reduced}. We can also assume that $\mathcal W_1$ and $\mathcal W_2$ do not share a common subweb (otherwise, $tang(\mathcal W_1,\mathcal W_2) = X$ and there is nothing to do).

Under that assumption the web $\mathcal W_1 \boxtimes \mathcal W_2$ is reduced and
$$tang(\mathcal W_1,\mathcal W_2) = \Delta(\mathcal W_1 \boxtimes \mathcal W_2).$$
So the corollary follows from Corollary \ref{discriminant-invariant} since $\mathcal W_1 \boxtimes \mathcal W_2$ is invariant (as a sum of two invariant webs).
\end{proof}
%
%

\subsection{Algebraic factors and invariant webs}  We start by recalling the relationship between rational factors and invariant foliations developed in \cite{moi3}. We will use the notions described in Section 1.5 of algebraically integrable foliations and foliations tangent to the fibres of a dominant rational map $\phi: X \dashrightarrow C$

%

{\Prop[{\citep[Proposition 3.1.5]{moi3}}]\label{extension}   Let $f : (X,v) \dashrightarrow (C,w)$ be a rational factor of dimension one of a smooth irreducible complex $D$-variety $(X,v)$. The foliation $\Theta_{X/C}$ is invariant under the vector field $v$.}    

Note that by definition the foliation $\Theta_{X/C}$ is algebraically integrable. So if $(X,v)$ admits a rational factor of dimension one then it admits an invariant algebraically integrable foliation. 

{\Prop\label{linearizationofalgebraicfactors} Let $(X,v)$ be a smooth irreducible complex $D$-variety. If $(X,v)$ admits an algebraic factor of dimension $1$ then for some $r \geq 1$, the $D$-variety $(X,v)$ admits an invariant  algebraically integrable $r$-web $\mathcal W$.}

\begin{proof}
Assume that $(X,v)$ admits an algebraic factor of dimension one. This means that there exists a generically finite extension $\phi: (X',v') \dashrightarrow (X,v)$ and a rational factor $f: (X',v') \dashrightarrow (C,w)$ of $(X',v')$.

Up to extending $X'$, we can assume that the extension $\mathbb{C}(X) \subset \mathbb{C}(X')$ is a Galois extension of fields and we denote by $G$ its Galois group. We can find a dense affine open set $U'$ in $X'$ such that the birational action of $G$ on $X'$ is regular on $U'$ and stabilizes $U'$ and such that  $U = U'/G$ is a dense open set of $X$ and the restriction of $\phi$ to $U$ is the quotient morphism 

$$\pi: U' \rightarrow U'/G = U.$$

By Proposition \ref{extension}, the foliation $\Theta_{X'/C}$ is an algebraically integrable foliation on $X'$ invariant by $v'$.  By Proposition \ref{pushforward-finitegroups}, there exists an $r$-web $\mathcal W_U$ on $U$  such that 
$$(\ast): \phi^\ast \mathcal W_U = \mathcal F_1 \boxtimes \ldots \boxtimes \mathcal F_r$$
where $\mathcal F_1,\ldots, \mathcal F_r$ are foliations on $X$ conjugated to $\Theta_{X'/C}$ under the action of $G$. By Proposition \ref{web-saturation}, the web $\mathcal W_U$ extends uniquely to a web denoted $\mathcal W$ on $X$. We claim that this web is algebraically integrable and invariant under $v$.
\begin{itemize}
\item Since  $\Theta_{X'/C}$ is an algebraically integrable foliation on $X'$,  $\mathcal F_1,\ldots, \mathcal F_r$ which are conjugate to $\Theta_{X'/C}$ under the action of $G$ are also algebraically integrable. Using $(\ast)$, this implies that $\mathcal W_U$ is algebraically integrable, so that $\mathcal W$ is algebraically integrable.

\item Note that $v$ is invariant under the action of $G$ (since it descends to a vector field on $X$). Since  $\Theta_{X'/C}$ is invariant under $v$, this implies that $\mathcal F_1,\ldots, \mathcal F_r$ are also invariant foliations. So using $(\ast)$, $\phi^\ast \mathcal W_U$ is an invariant web on $X'$. 
\end{itemize}

Consider a local generator $\omega$  of $\mathcal W_U$ then  $\phi^\ast \omega$ is a generator of $\phi^\ast \mathcal W_U$ and
$$ \phi^\ast (\mathcal L_v(\omega) \wedge \omega) = \mathcal L_{v'}( \phi^\ast\omega) \wedge \phi^\ast \omega = 0$$

Since $\phi$ is dominant, $\phi^\ast$ is injective and $ \mathcal L_v(\omega) \wedge \omega = 0$. It follows that $\mathcal W_U$ and hence $\mathcal W$ is invariant.
\end{proof}

\subsection{Proof of Theorem \ref{maintheorem}} Let $(X,v)$ be a  smooth complex $D$-variety of dimension two satisfying the hypotheses of Theorem \ref{maintheorem}. 

We want to show that if $\mathcal W$ an $r$-web on $X$ invariant under the vector field $v$ then $\mathcal W = \mathcal F(v) \boxtimes \cdots \boxtimes \mathcal F(v)$. We prove it by induction or $r \geq 1$:
\begin{itemize}
\item If $r = 1$, then $\mathcal W$ is an invariant foliation $\mathcal F$ on $X$. If $\mathcal F \neq \mathcal F(v)$, by Lemma \ref{tangency-foliations}, $tang(\mathcal F,\mathcal F(v))$ is of pure codimension one.
\end{itemize}
Moreover, by definition, $p \in tang(\mathcal F,\mathcal F(v))$ so that $tang(\mathcal F,\mathcal F(v))$ is a complex algebraic curve on $X$ containing $p$. By Corollary \ref{invariance-tangency}, the algebraic curve  $tang(\mathcal F,\mathcal F(v))$  is invariant under the vector field $v$. This contradicts the assumption made on $p$.

\begin{itemize}
\item Assume now that $r \geq 1$ and that the result holds for all $r$-webs and consider $\mathcal W$ an invariant $(r+1)$-web on $X$. Like for the case $r = 1$,  $p \in tang(\mathcal W,\mathcal F(v))$ and  $tang(\mathcal W,\mathcal F(v))$ is a closed subvariety of $X$ invariant under $v$.
\end{itemize}
By Proposition \ref{tangency-webs}, either $tang(\mathcal W,\mathcal F(v))$  has  pure codimension one or $\mathcal W$ and $\mathcal F(v)$ share a common subweb. The first case can not hold by assumption on $p$ and as $\mathcal F(v)$ (like any foliation) is irreducible as a web, the second case implies that

$$\mathcal W = \mathcal W_1 \boxtimes  \mathcal F(v)$$
for some $r$-web $\mathcal W_1$ on $X$. Since $\mathcal W$ is invariant, Proposition \ref{invariance for webs} implies that $\mathcal W_1$ is also invariant, since the irreducible components of the horizontal divisor associated to $\mathcal W_1$ are invariant. So we can apply the induction hypothesis to $\mathcal W_1$.

This finishes the proof of the first part of Theorem \ref{maintheorem}. To prove the second part,  we first note that the foliation $\mathcal F(v)$ is not algebraically integrable:

Indeed, assume that $\mathcal F(v)$ is algebraically integrable.  By the theorem of Camacho and Sad \cite{Camacho-Sad}, the singularity $p$ always admits at least one analytic separatrix, i.e. $p$ is contained in an analytic curve $C$ tangent to the foliation $\mathcal F(v)$.  Since $\mathcal F(v)$ is algebraically integrable, the Zariski-closure $\overline{C}$ of $C$ is an algebraic curve containing $p$. Since $\overline{C}$ is the Zariski closure of a curve tangent to $v$, it is also invariant under $v$ (see Proposition 3.1.23 of \cite{moi}), which contradicts our hypothesis that $p$ is not contained in any algebraic curve tangent to $v$.

\vspace{0.3cm}

Since $\mathcal F(v)$ is not algebraically integrable, the first part of the theorem implies that $(X,v)$ does not admit any invariant algebraically integrable web.  Hence, by Proposition \ref{linearizationofalgebraicfactors}, $(X,v)$ does not admit any algebraic factor of dimension one. This proves the second part of Theorem \ref{maintheorem}. \qed
\section{Disintegration and minimality of planar algebraic vector fields}

\subsection{Differentially closed fields} We work in the ambient theory $\textbf{DCF}_0$ of differentially closed fields of characteristic $0$ and \textbf{we fix $(\mathcal U, \delta)$ a saturated model}. It is well-known that the theory $\textbf{DCF}_0$ is $\omega$-stable and admits the elimination of quantifiers and imaginaries in the language $\mathcal L_\delta = \lbrace 0,1,+,\times,- \rbrace$ of differential fields. See Chapter 2 of \cite{MTF} for a self-contained presentation.

We will use the formalism of Zariski geometries to present definable sets of finite rank living in $\U$ as described in Section 1 of \cite{Itai}:  If $X$ is an algebraic variety defined over the field of constants $\U^\delta$ of  $\U$, the derivation $\delta$ induces a definable map
$$ \nabla: X(\U) \rightarrow T_X(\U).$$

If $X \subset \mathbb{A}^n$ is an affine variety then $\nabla$ is defined by $$\nabla(x_1,\ldots,x_n) = (x_1,\ldots x_n, \delta(x_1),\ldots \delta(x_n)) \in X(\U) \times \U^n $$
and one easily checks (using the Leibniz rule) that $\nabla$ takes values in the closed subset  $T_X(\U) \subset X(\U) \times \U^n$. More generally, this construction can be carried out  for an arbitrary algebraic variety $X$ by covering $X$ by affine open sets.

Note that the point of view of Section 1 of \cite{Itai} is more general and that we focus here on \textit{autonomous} differential equations or in other words differential equations defined over constant differential fields $k \subset \U^\delta$ where $\U^\delta = \lbrace x \in \U \text{ | } \delta(x) = 0\rbrace.$

{\defn Let $k \subset \U^\delta$ be a subfield, $X$ be an algebraic variety over $k$ and $v$ be a vector field on $X$ defined over $k$. The $k$-definable set associated to $(X,v)$ denoted $(X,v)^\delta$ is given by the equalizer:

$$(X,v)^\delta = \lbrace x \in X(\U) \text{ | } \nabla(x) = v(x)\rbrace.$$}

{\Prop \label{geometric description} Let $k \subset \U^\delta$ be a field of characteristic $0$, $X$ be an algebraic variety over $k$ and $v$ be a vector field on $X$ defined over $k$. 
\begin{itemize}
\item[(i)] Any $k$-definable subset of $(X,v)^\delta \times \ldots \times (X,v)^\delta$ is a Boolean combination of  subsets of the form $(Z,v_{|Z})^\delta$ where $Z$ a closed variety of $X \times \ldots \times X$ invariant under the vector field $v \times \ldots \times v$.

\item[(ii)] If $X$ is irreducible then $(X,v)^\delta$ contains a unique type denoted $p_{(X,v)} \in S(k)$ of maximal transcendence degree. It is the unique type $p(x) \in S(k)$ satisfying:
$$x \in (X,v)^\delta \text{ and } x \notin (Z,v_{|Z})^\delta  \text{ for every proper closed invariant } Z \subset X.$$

\item[(iii)] Any type $p \in S(k)$ of finite rank is interdefinable with a type of the form $p_{(X,v)}$ for some $k$-irreducible $D$-variety $(X,v)$.

\item[(iv)] Let $(Y,w)$ be another $k$-irreducible $D$-variety. Denote by $S$ the set of germs of $k$-definable maps sending $p_{(X,v)}$ to $p_{(Y,w)}$. Then there is a natural correspondence   
$$ S \simeq \lbrace \phi: X \dashrightarrow Y \text{ rational and dominant with }  d\phi(v) = w \rbrace.$$
\end{itemize}}

See Proposition 1.1 and Proposition 1.2 of \cite{Itai}.

{\defn Let $X$ be an algebraic variety over some field $k$ of characteristic $0$, endowed with a vector field $v$. The \textit{geometric language associated to $(X,v)$} denoted $\mathcal L_{(X,v)}$ is simply the relational language with one  $n$-ary predicate $R_Z$ for each closed irreducible invariant subvariety $Z$ of $(X,v)^n$.}

\vspace{0.1cm}

Note that $(X,v)^\delta$ can either be seen as a definable set of the theory $\textbf{DCF}_0$ or as an abstract structure in the geometric language where the interpretation of the $n$-ary predicate  $R_Z$ is given by 

$$ (Z,v_{|Z})^\delta \subset (X,v)^\delta \times \cdots \times (X,v)^\delta.$$

Proposition \ref{geometric description} shows that these two points of view are essentially equivalent. In particular $(X,v)^\delta$ is strongly minimal and disintegrated as a definable subset of $\U$ if and only if it has the same properties viewed as an abstract structure in the geometric language.

\vspace{0.1cm}
The main theorem of this section is the following result which uses Theorem \ref{maintheorem} together with structural results of the theory $\textbf{DCF}_0$ and its definable sets from the 90s  (for instance \cite{Sok}).

{\Thm[Model-theoretic version]\label{Maintheorem2} Let $d \geq 3$ and let $v$ be a planar algebraic vector field of degree $d$ with $\mathbb{Q}$-algebraically independent coefficients.

The set of solutions of the differential equation $(\mathbb{A}^2,v)$ in a differentially closed field is strongly minimal and disintegrated.} 

{\rem The conclusion of Theorem \ref{Maintheorem2} does not hold for any planar algebraic vector field of degree $d = 1$, since they all define differential equations which are linear and hence of Morley rank two, internal to the constants. However, it is possible that, similarly to Theorem \ref{Ilyashenko}, Theorem \ref{Maintheorem2} in fact holds for all degrees $d \geq 2$.

Using the results of this article, the version of Theorem \ref{Maintheorem2} in degree $d = 2$ is equivalent to the existence of \textit{one} planar algebraic vector field of degree $2$ orthogonal to the constants.}

 \subsection{Families of $D$-varieties with constant coefficients}  We first gather some properties of families of autonomous differential equations and describe what follows from the structure of the algebraically closed field $\U^\delta$ of constants.

{\defn Let $k \subset \U^\delta$ be a subfield.  We call a \textit{$k$-algebraic family of varieties endowed with vector fields with constant coefficients}, any surjective morphism $\pi: (\mathcal X, v) \longrightarrow (S,0)$ in the category of algebraic varieties endowed with vector fields over $k$.}

We fix $\pi: (\mathcal X, v) \longrightarrow (S,0)$ a  $k$-algebraic family of varieties endowed with vector fields with constant coefficients. 
If $s \in S(\U^\delta)$, the fibre of $\pi$ over $s$ denoted $(\mathcal X,v)_s$ is an autonomous differential equation with parameters in $k(s)$.  We consider the property $\mathcal P(s)$ given by
\begin{center}
$\mathcal P(s)$: the set of solutions $(\mathcal X,v)^\delta_s$ of $(\mathcal X,v)_s$ is strongly minimal disintegrated.
 \end{center}

{\lem\label{lemmequivalence} Let $k$ be a countable subfield of $\U^\delta$ and let $\pi: (\mathcal X, v) \longrightarrow (S,0)$ be a $k$-algebraic family of varieties endowed with vector fields with constant coefficients.

If $s,t$ are two elements of $S(\U^\delta)$ realizing the same type (in the sense of $\textbf{ACF}_0$) over $k$ then $\mathcal P(s)$ holds if and only if $\mathcal P(t)$ holds.} 

\begin{proof}
Consider two elements $s,t \in S(\U^\delta)$ which realize the same type (in $\textbf{ACF}_0$) over $k$. Since the field of constants of $(\U,\delta)$  is a pure algebraically closed fields, the elements $s,t \in S(\mathbb{C})$ also satisfy the same type over $k$ in the differentially closed field $\U$. 

By saturation of $(\U,\delta)$, it follows that there exists an automorphism $\sigma$ of $(\U,\delta)$ fixing $k$ and such that $\sigma(s) = t$. It follows that

$$\sigma((\mathcal X,v)_s^{\delta}) = (\mathcal X,v)_{\sigma(s)}^{\delta} = (\mathcal X,v)_{t}^{\delta}.$$

Since $\sigma$ is an automorphism of $\U$, for every definable set $D$, $D$ is strongly minimal and disintegrated if and only if $\sigma(D)$ has the same properties. Hence $\mathcal P(s)$ holds if and only if $\mathcal P(t)$ holds.
\end{proof}

{\Prop \label{equivalence} Let $k$ be a countable subfield of $\mathbb{C}$ and  $\pi: (\mathcal X, v) \longrightarrow (S,0)$ be a $k$-algebraic family of varieties endowed with vector fields with constant coefficients.

Assume that $S$ is irreducible, then the following are equivalent:
\begin{itemize}
\item[(i)] The property $\mathcal P(s)$ holds on a measurable set of parameters $s \in S(\mathbb{C})$ of positive Lebesgue measure.
\item[(ii)] The property $\mathcal P(s)$ holds on a somewhere dense $G_\delta$-set  of $S(\mathbb{C})^{an}$: the property $\mathcal P(s)$ holds on a countable intersection $A = \bigcap_{i \in \mathcal I} U_i$ of open sets $U_i$ of  $S(\mathbb{C})^{an}$ such that $\overline{A}$ has non empty interior. 
\item[(iii)] The property $\mathcal P(s)$ holds for some/any realization $s$ of the generic type of $S$ over $k$ (in the sense of $\textbf{ACF}_0$). 
\end{itemize}}

\begin{proof}
$(i) \Rightarrow(iii)$ Since $k$ is countable, there are countably many non-generic types living on $S(\mathbb{C})$. For each of these types $p$, the set of realizations of $p$ in  $S(\mathbb{C})$ is contained in a proper closed algebraic subvariety of $S(\mathbb{C})$ of Lebesgue measure zero.

So the set of non-generic elements of  $S(\mathbb{C})$ (over $k$) has Lebesgue measure zero. It follows that any measurable set of positive Lebesgue measure contains a generic element.

$(ii) \Rightarrow(iii)$ Since $k$ is countable, the set of generic elements (over $k$) is a countable intersection of dense open sets hence it is a everywhere dense $G_\delta$-set of $S(\mathbb{C})^{an}$.

As such, it must intersect any somewhere dense $G_\delta$-set of $S(\mathbb{C})^{an}$ so that  any somewhere dense $G_\delta$-set of $S(\mathbb{C})^{an}$ contains a generic element of $S(\mathbb{C})$.

$(iii) \Rightarrow (i),(ii)$ follows from Lemma \ref{equivalence} since the set of generic elements in $S(\mathbb{C})$ over $k$ is an everywhere dense $G_\delta$-set and its complement has  Lebesgue measure zero.
\end{proof}

{\defn Let $d \geq 1$. We say that a complex vector field $v = f(x,y) \frac \partial {\partial x} + g(x,y) \frac \partial {\partial y}$ of degree $d$ has \textit{$\mathbb{Q}$-algebraically independent coefficients} if after writing
$$f(x,y) = \sum_{i+ j \leq d} f_{i,j}x^iy^j \text{ and }g(x,y) = \sum_{i+ j \leq d} g_{i,j}x^iy^j,$$
the complex coefficients $f_{i,j}$ and $g_{i,j}$ for $i + j \leq d$ are all distinct, non zero and $\mathbb{Q}$-algebraically independent.} 

{\cor\label{reduction} Theorem \ref{Maintheorem2} holds if and only if it holds for a specific planar algebraic vector field of degree $d$ with $\mathbb{Q}$-algebraically independent coefficients.} 

\begin{proof}
Denote by $S_d \simeq \mathbb{A}_\mathbb{Q}^{(d + 1)(d+2)}$ the parameter space for vector fields of degree $\leq d$ and consider the $\mathbb{Q}$-algebraic family of varieties endowed with vector fields with constant coefficient given by all vector fields of degree $\leq d$:
$$\pi : (\mathcal X = \mathcal{A}^2 \times S_d , v_d) \longrightarrow (S_d,0).$$

In other words, $v_d$ is is an algebraic vector field on $\mathcal X$, tangent  to the fibres of $\pi$ such that for every complex point $s \in S_d(\mathbb{C})$, the fibre $(\mathcal X,v_d)_s$ is isomorphic (via the second projection on $\mathbb{A}^2$) to $(\mathbb{A}^2,v_s)$.

Lemma \ref{lemmequivalence} applied to this family shows that Theorem \ref{Maintheorem2} holds if and only if it holds for a realisation of the generic type of $S_d$ over $\mathbb{Q}$, namely for a planar algebraic vector field of degree $d$ with $\mathbb{Q}$-algebraically independent coefficients.
\end{proof}

\subsection{Complex invariant curves} The behavior of complex invariant curves for a complex planar vector field with $\mathbb{Q}$-algebraically independent coefficients is described by the work of Petrovskii and Landis in the 50s:

{\Thm[\cite{Petro}; {\citep[Theorem 25.56]{Ily}}]\label{Ilyashenko} Let $d \geq 2$. A complex planar vector field of degree $d$ with $\mathbb{Q}$-algebraically independent coefficients admits no invariant algebraic curve (excepted the line at infinity which is invariant by the foliation generated by $v$).}

\subsection{Minimality via algebraic factors}  We say that a complex $D$-variety $(X,v)$ \textit{does not admit non-trivial rational factors} if $(X,v)$ is a generically finite extension of every rational factor of $(X,v)$ of positive dimension. We say that $(X,v)$ \textit{does not admit non-trivial algebraic factors} if every generically finite extension of $(X,v)$ does not admit non-trivial rational factors.

The heart of our strategy to the proof of Theorem \ref{Maintheorem2} is the following proposition:

{\Prop\label{modeltheoryreduction} Let $k$ be a subfield of $\U^\delta$ and $ (X,v)$ be an absolutely irreducible $D$-variety over $(k,0)$ of positive dimension.
\begin{itemize}
\item[(i)] If the $D$-variety $(X,v)$ does not admit non-trivial rational factors then its generic type is semi-minimal.
\item[(ii)] If the $D$-variety $(X,v)$ does not admit non-trivial algebraic factors and if its generic type is orthogonal to the constants then its generic type is minimal.
\end{itemize}}

\begin{proof}

We work inside a saturated differentially closed field $\U$ for the proof. Denote by $p \in S(k)$ the generic type of $(X,v)$ and denote by $a$ a realization of $p$ inside $\U$. \\

(i). Since $p$ is a stationary type of finite U-rank which is not algebraic, there exists an extension of the parameters $(k,0) \subset (l,\delta_l)$ and a minimal type $q \in S(l)$ such that $p$ is non-orthogonal to $q$ by \citep[Lemma 5.1, Chapter 2]{GST}.

Consider the set $\mathcal Q$ of types (with parameters in  $\U$) conjugate to the type $q$ over $k$. The set of types $\mathcal Q$ is $k$-invariant and hence by the Decomposition Lemma \citep[Lemma 4.5, Chapter 7]{GST}, there exists $a_0 \in \mathrm{dcl}(a,k) \setminus \mathrm{acl}(k)$ such that $\mathrm{tp}(a_0/k)$ is internal to $\mathcal Q$.

Consider a $D$-variety $(Y,w)$ over $(k,0)$ such that its generic type is interdefinable with  $\mathrm{tp}(a_0/k)$. Since $a_0 \in \mathrm{dcl}(a,k)$ and $a$ realizes the type $p$, there exists a rational dominant morphism of $D$-varieties $\pi: (X,v) \dashrightarrow (Y,w)$ (see Proposition \ref{geometric description}). But $(X,v)$ does not admit non-trivial rational factors, so $\pi$ has a finite generic fibre or in other words $a \in \mathrm{acl}(a_0,k)$. It follows that $\mathrm{tp}(a/k)$ is almost internal to the set $\mathcal Q$ of conjugates of the minimal type $q$ and therefore semi-minimal. \\

 (ii). We already know by (i) that the type $p$ is semi-minimal. Moreover, by assumption $p$ is orthogonal to the constants. By the canonical base property, it follows that $p$ is analyzable by locally modular types and therefore one-based (see \citep[Fact 1.3]{Chat} for example). 

{\asser Let $p \in S(k)$ be a one based stationary type which is not algebraic. There exists a minimal type $q \in S(k)$ such that $p$ and $q$ are not almost orthogonal (over $k$).}

\begin{proof}
Consider $a \models p$ and an extension of parameters $k \subset l$ such that $U(\mathrm{tp}(a/l)) = U(\mathrm{tp}(a/k)) - 1$ and $\mathrm{tp}(a/l)$ is stationary and set $e$ for the canonical base of $\mathrm{tp}(a/l)$. By construction, since $p$ is one-based, $e \in \mathrm{acl}(a,k)$.  As e $\notin acl(k)$, the type $q = \mathrm{tp}(e/k)$ is not almost orthogonal to $p$. Moreover, we claim that $U(q) = 1$. Indeed:
$$ U(a/k) = U(e,a/k)  = U(a,e/e,k) + U(e/k) = U(a/k) - 1 +  U(e/k). \qedhere$$
\end{proof}

Consider such a minimal type $q \in S(k)$ and $b \models q$. Since the type $q$ is minimal and not almost orthogonal to $p$, we have that $b \in \mathrm{acl}(k,a)$. Consider a $D$-variety $(X',v')$ over $(k,0)$ such that its generic type is interdefinable with $\mathrm{tp}(a,b/k)$ and $(Y,w)$ a $D$-variety over $(k,0)$ such that its generic type is interdefinable with  $\mathrm{tp}(b/k)$ (using Proposition \ref{geometric description} (iii) again).

Since $b \in \mathrm{acl}(a,k)$, the $D$-variety $(X',v')$ is a generically finite extension of $(X,v)$ and by construction $(Y,w)$ is a rational factor of $(X',v')$. But $(X,v)$ does not admit any non-trivial algebraic factors. Hence, the generic fibre of the rational factor  $(X',v') \dashrightarrow (Y,w)$ is finite so that $a \in \mathrm{acl}(b,k)$. It follows that $U(a/k) = U(b/k) = 1$ and therefore $p$ is minimal.
\end{proof}

\subsection{Disintegration of semi-minimal types with constant parameters} The results of this section are consequences, in geometric terms, of the Hrushovski-Sokolovic Theorem (cf \cite{Sok}). We will work over a subfield $k$ of the field $\U^\delta$ of constants.

{\nota Let $v$ be a vector field on an algebraic variety $X$ over $k$. For every $n \geq 2$, we denote by $\mathcal I_n(X,v)$ the set of invariant irreducible subvarieties of the product $D$-variety $\underbrace{(X,v) \times \cdots \times (X,v)}_\text{n times}$, which project dominantly on each factor.}

{\lem[{\citep[Appendix A]{moi}}]\label{simple properties} Let $v$ be a vector field on an algebraic variety $X$ over $k$ and $\pi: X^{m} \rightarrow X^{n}$ any projection on coordinates with $m \geq n$.

\begin{itemize}
\item[(i)] If $Z,Z'$ are two elements of $\mathcal I_n(X,v)$, then the irreducible components of the intersection $Z \cap Z'$ are elements of $\mathcal I_n(X,v)$.
\item[(ii)] If $Z$ is an element of $\mathcal I_m(X,v)$, then the Zariski-closure $\overline{\pi(Z)}$ of $\pi(Z)$ is an element of $\mathcal I_n(X,v)$.
\item[(iii)] If $Z$ is an element of $\mathcal I_n(X,v)$ then the irreducible components of $\pi^{-1}(Z)$ are elements of $\mathcal I_n(X,v)$.
\end{itemize}}

The following definition makes sense for any collection of algebraic subvarieties of $X^n$ (for some algebraic variety $X$) satisfying the conclusion of Lemma \ref{simple properties}:

{\defn Let $v$ be a vector field on an algebraic variety $X$ over $k$. We say that the sequence $(\mathcal I_n(X,v))_{n \in \mathbb{N}}$ is \textit{eventually constant} if there exists $r \geq 2$ such that any element $Z \in \mathcal I_n(X,v)$ can be written as an irreducible component of an intersection:

$$ \bigcap_{E \in P_r(n)} \pi_E^{-1}(Z_E).$$ 
where $E$ runs over the subsets of $r$ elements of $\lbrace 1, \cdots, n \rbrace$, $\pi_E$ denotes the projections on the coordinates in $E$ and $Z_E$ is an element of $\mathcal I_r(X,v)$.}

{\defn \label{generically-disintegrated} Let $v$ be a vector field on an algebraic variety $X$ over $k$. We say that $(X,v)$ is \textit{generically disintegrated} if the sequence $(\mathcal I_n(X,v))_{n \in \mathbb{N}}$ is eventually constant starting at $r = 2$.}

{\Thm\label{Hrushovski-Sokolovic consequences} Let $v$ be a vector field on an algebraic variety $X$ over $k$. Assume that the generic type of $(X,v)$ is semi-minimal. The following are equivalent:
\begin{itemize}
\item[(i)] The differential equation $(X,v)$ is generically disintegrated.

\item[(ii)] The sequence $(\mathcal I_n(X,v))_{n \in \mathbb{N}}$ is eventually constant.

\item[(iii)] For all $n \in \mathbb{N}$, $(X,v)^n$ does not admit any non-constant rational integral. 
\end{itemize}}

\begin{proof}
Clearly, $(i) \Rightarrow (ii) \Rightarrow (iii)$. The implication $(iii) \Longrightarrow (i)$ is proved using Hrushovski-Sokolovic in the first section of \cite{moi}.
\end{proof}

\subsection{Orthogonality to the constants} 

{\defn  Let $v$ be a vector field on an algebraic variety $X$ over some field $k$ of characteristic $0$. We say that $(X,v)$ is \textit{orthogonal to the constants} if the condition (iii) of Theorem \ref{Hrushovski-Sokolovic consequences} is fulfilled.}

{\Thm[\cite{moi}]\label{orthogonality} Let $d \geq 3$. A complex planar vector field of degree $d$ with $\mathbb{Q}$-algebraically independent coefficients is orthogonal to the constants.}
\vspace{0.2cm}

This theorem is the only part in the proof of Theorem \ref{introa} where we actually use that $d \geq 3$ and not only that $d \geq 2$. Using a specialization argument, we show in \cite{moi} that to prove the theorem in degree $d \geq d_0$, it is enough to construct \textit{one} complex planar algebraic vector field of degree $d_0$ which satisfies the conclusion of Theorem \ref{orthogonality}.

For $d_0 = 3$, we proceed as follows: consider the vector field $v(x) = x^2(x-1) \frac \partial {\partial x}$ on $\mathbb{A}^1$. Since $\frac 1 {x^2(x-1)}$ has both a double pole and a simple one, it is neither the logarithmic derivative nor the derivative of a rational function in $(\mathbb{C}(x),\frac \partial {\partial x})$. The main theorem of \cite{Ros} implies that $(\mathbb{A}^1,v)$ is orthogonal to the constants. If follows that  the planar vector field of degree $3$:
$$w(x,y) = x^2(x-1) \frac \partial {\partial x} + y^2(y-1) \frac \partial {\partial y}$$
has the same property. Note that a similar construction is not possible when $d_0 = 2$.

\subsection{Proof of Theorem \ref{Maintheorem2}}  Let $d \geq 3$ and let $v(x,y) = f(x,y) \frac \partial {\partial x} + g(x,y) \frac  \partial {\partial y}$ be a complex algebraic vector field of degree $d$ with $\mathbb{Q}$-algebraically independent coefficients.

The set of zeros of the vector field $v$ is defined by $$f(x,y) = g(x,y) = 0.$$ So it is a finite subset of  $\mathbb{C}^2$ of cardinal $d^2$ by assumption on $f(x,y)$ and $g(x,y)$. We denote by $p \in \mathbb{C}^2$ one of the zeros of $v(x,y)$,  by $F \subset \mathbb{C}$ the finitely generated subfield of the complex numbers generated by the coefficients of $v$ and by $q \in S(F)$ the generic type of $(\mathbb{A}^2,v)$ over $F$.

{\asser The type $q$ is minimal and disintegrated.}

 \begin{proof}
 By Theorem \ref{orthogonality}, the type $q$ is orthogonal to the constants. Moreover, by Theorem \ref{Ilyashenko}, the vector field $v$ does not admit any invariant algebraic curve. In particular, the zero $p \in \mathbb{C}^2$ has no algebraic separatrix. So the hypotheses of Theorem \ref{maintheorem} are fulfilled and we conclude that $(\mathbb{A}^2,v)$ has no algebraic factors of dimension one. 
 
The first part of the claim now follows from Proposition \ref{modeltheoryreduction} (ii) while the second part follows from  Proposition \ref{modeltheoryreduction} (i) and  Theorem \ref{Hrushovski-Sokolovic consequences}. \end{proof}

{\asser[{\citep[Theorem 6.1]{Moosa-Freitag}}] Let $A \subset \U$ and $q \in S(A)$ be a type of order $\leq 2$. Then the Lascar rank and the Morley rank of $q$ agree.}
\vspace{0.1cm}

The claim follows from the finiteness of the number of invariant hypersurfaces for a vector field without non constant rational integral (Jouanolou Theorem, see \cite{Joua}). It was noticed first by Marker and Pillay in the case of type of order $\leq 2$ over constant differential field and later generalized by Freitag and Moosa to arbitrary types of order $\leq 2$.

Using the two previous claims, we obtain that $q$ is strongly minimal and disintegrated. To conclude that the set of solutions $(\mathbb{A}^2,v)^\delta$ in a differentially closed field $(\U,\delta)$ is strongly minimal and disintegrated, it remains to study the non-generic behavior:

{\asser With the notation above, denote by $\mathcal Q$ the set of realizations of $q$ in $\mathcal U$. Then $(\mathbb{A}^2,v)^\delta \setminus \mathcal Q$ is a finite subset of $\mathrm{acl}(F)$.}

\begin{proof}
Let $r \in S(F)$ be a type living on $(\mathbb{A}^2,v)^\delta$ which is not the generic type. By Proposition \ref{geometric description}, $r$ is the generic type of a proper irreducible closed invariant subvariety $Z$ of $(\mathbb{A}^2,v)$.  By applying Theorem  \ref{Ilyashenko} again, we see that $v$ does not admit any closed invariant algebraic curve. It follows that $\mathrm{dim}(Z) = 0$ so that $Z$ is a finite union of invariant points i.e. complex zeros of the vector field $v$. We conclude that $$(\mathbb{A}^2,v)^\delta \setminus \mathcal Q \subset \lbrace (x,y) \in \mathbb{\U}^2, f(x,y) = 0, g(x,y) = 0 \rbrace \subset \mathrm{acl}(F)$$
is a finite subset of $\mathrm{acl}(F)$.
\end{proof}

We have therefore shown that $(\mathbb{A}^2,v)^\delta$ is the union of the set of realizations in $\U$ of a strongly minimal disintegrated type and of a finite set included in $\mathrm{acl}(F)$. So $(\mathbb{A}^2,v)^\delta$ is strongly minimal and disintegrated. \qed

\section{Two consequences of Theorem \ref{Maintheorem2}}
\subsection{Minimality and irreducibility in the sense of Nishioka-Umemura} Following Umemura, we recall the definition of a differential fields of $1$-classical functions (see also the appendix of \cite{Nag2} for a model-theoretic translation of this property).

{\defn We say that a differential field $(K,\delta)$ is a \textit{differential field of $1$-classical functions} if the field of constants of $(K,\delta)$ is the field of complex numbers and if there exists a tower of differential fields:
$$(K_0, \delta_0)  = (\mathbb{C}(t), \frac d {dt}) \subset (K_1, \delta_1) \subset \ldots  \subset  (K_n, \delta_n) = (K,\delta)$$
such that for every $i \leq n - 1$, the differential field extension $(K_i, \delta_i) \subset (K_{i + 1}, \delta_{i + 1}) $ satisfies one of the following properties:
\begin{itemize}  
\item The extension $(K_i, \delta_i) \subset (K_{i + 1}, \delta_{i + 1})$ is a finite algebraic extension.
\item The extension $(K_i, \delta_i) \subset (K_{i + 1}, \delta_{i + 1})$ is a strongly normal extension.
\item The extension $(K_i, \delta_i) \subset (K_{i + 1}, \delta_{i + 1})$  is finitely generated and of transcendence degree $1$. In other words, there exists an integral integral $D$-curve $(C_,\delta_C)$ over $(K_i,\delta_i)$ such that $(K_{i + 1}, \delta_{i + 1})$ is the function field of  $(C_,\delta_C)$ .
\end{itemize}}

\nota Let $(K,\delta)$ be a differential field extension of $(\mathbb{C},0)$. If $v = f(x,y) \frac {\partial} {\partial x} + g(x,y) \frac {\partial} {\partial y}$ is a complex planar vector field, we denote $(\mathbb{A}^2,v)^{(K,\delta)}$ the set of solutions of the differential equation $(\mathbb{A}^2,v)$ in $(K,\delta)$. In other words, 
$$(\mathbb{A}^2,v)^{(K,\delta)} = \lbrace (x,y) \in K^2 \text{ | } \delta(x) = f(x,y) \text{ and } \delta(y) = g(x,y) \rbrace.$$

The previous definition can be generalized to arbitrary complex $D$-varieties $(X,v)$. In fact, if $(\U,\delta)$ is a differentially closed field extending $(K,\delta)$, then
$$ (X,v)^{(K,\delta)}= (X,v)^\delta \cap \mathrm{dcl}(K)$$
using the notation from Section 3.1.

{\rem In this language, the theorem of Landis-Petrovskii (see Theorem \ref{Ilyashenko}) can be expressed in the following form: for every general planar algebraic vector field $v$ and every differential field extension $(K,\delta)$ over $(\mathbb{C},0)$, finitely generated and of transcendence degree $1$, we have that  
$$(\mathbb{A}^2,v)^{(K,\delta)} = (\mathbb{A}^2,v)^{(\mathbb{C},0)} = \lbrace \text{stationary solutions  of }  (\mathbb{A}^2,v)  \rbrace.$$}

The following consequence of Theorem \ref{Maintheorem2} shows that it expresses a stronger non-integrability statement for vector fields of degree $d \geq 3$:

{\cor Let $v$ be a general planar algebraic vector field of degree $d \geq 3$ and let $(K,\delta)$ be any field of $1$-classical functions. Then:
$$(\mathbb{A}^2,v)^{(K,\delta)} = (\mathbb{A}^2,v)^{(\mathbb{C},0)} = \lbrace \text{stationary solutions  of }  (\mathbb{A}^2,v)  \rbrace.$$}

It follows from Proposition A.8 of \cite{Nag2} but we repeat the proof for the sake of completeness.
\begin{proof}
We work inside a saturated differentially closed field $\U$. Consider $(K,\delta)$ a sub differential field of $1$-classical functions given as a tower of differential fields:
$$(\mathbb{C},0) = (K_0,\delta_0)  \subset  (\mathbb{C}(t), \frac d {dt}) = (K_1,\delta_1) \subset \ldots \subset (K_n,\delta_n) = (K,\delta).$$

For every $i \leq  n$, consider a tuple $x_i$ such that $K_{i+1} = K_i(x_i)$ as fields; therefore $K = \mathbb{C}(t,x_1,\ldots, x_n)$. We claim that $\mathrm{tp}(t,x_1,\ldots,x_n/\mathbb{C})$ is analyzable in the constants together with the set of all order $1$ differential equations (with arbitrary parameters): indeed, $\mathrm{tp}(t/\mathbb{C})$ is internal to the constants and  for every $i \leq n$, $\mathrm{tp}(x_i/t,x_1,\ldots, x_{i - 1})$ is either an algebraic type, a type internal to the constants or the generic type of an order one differential equation.

By Theorem \ref{Maintheorem2}, the generic type $p$ of $(\mathbb{A}^2,v)$ is strongly minimal of order $2$, orthogonal to the constants. Since two minimal non-orthogonal types have the same order, it follows that $p$ is orthogonal to every type of order $1$ (with arbitrary parameters). Hence,  the type $p$ is orthogonal to any type analyzable in the constants together with the set of  all order $1$ differential equations.

In particular, the types  $\mathrm{tp}(t,x_1,\ldots, x_n/ \mathbb{C})$ and $p$ are orthogonal so that $$\mathbb{C}(t,x_1,\ldots,x_n)^{alg} \cap p(\U) = \emptyset$$

It follows that $ (\mathbb{A}^2,v)^{(K,\delta)} $ is included in the set of non-generic solutions of $(\mathbb{A}^2,v)$ which are the stationary solutions by Theorem \ref{Ilyashenko}.
\end{proof}

\subsection{Disintegration and algebraic independence of the solutions}  The theorem of Landis and Petrovskii gives a full characterization of algebraic solutions for a general algebraic vector field of degree $d \geq 2$.  We consider the following generalizations of this result:

{\defn Let $v$ be an algebraic vector field on a complex algebraic variety $X$ and let $n$ be an integer. We say that the vector field \textit{$v$ satisfies the property $(C_n)$} if  every distinct analytic, non stationary, solutions $\gamma_1,\ldots, \gamma_n$  of $(X,v)$ are $\mathbb{C}$-algebraically independent.}
\vspace{0.2cm}

The Landis-Petrovskii Theorem (Theorem \ref{Ilyashenko}) shows that $(C_1)$ holds for general planar algebraic vector fields of degree $d \geq 2$. We show that for $n \geq 2$, the implication  $(C_2) \Rightarrow (C_n)$ holds for generic vector fields of degree $d \geq 3$.

{\defn\label{desintegration} Let $X$ be an algebraic variety over some field $k$ of characteristic $0$, endowed with a vector field $v$. We say that the vector field $v$ is \textit{disintegrated} if any irreducible closed invariant subvariety of $(X,v) ^n$ for some $n \geq 3$ can be written as an irreducible component of
$\bigcap_{i \neq j} \pi_{i,j}^{-1}(Z_{i,j})$
where, for all $i \neq j$, $\pi_{i,j}$ denotes the projection on the coordinates $i$ and $j$ and $Z_{i,j}$ is an irreducible closed invariant subvariety of $(X,v) \times (X,v)$.}
\vspace{0.1cm}

We refer to \cite{moi} and section 2.5 for a more detailed study of this notion in the language of algebraic varieties with vector fields. The following corollary states that the disintegration property is a typical property for vector fields of degree $d \geq 3$:

{\cor Let $d \geq 3$ and $v$ a general planar algebraic vector field of degree $d$.  Then $(\mathbb{A}^2,v)$ satisfies the disintegration property. In particular, for $n \geq 2$ the implication $(C_2) \Rightarrow (C_n)$ holds for a generic planar vector field of degree $d \geq 3$.}

\begin{proof}
By Theorem \ref{Maintheorem2}, the generic type of $(\mathbb{A}^2,v)$ is disintegrated and minimal. It follows that $(\mathbb{A}^2,v)$ is generically disintegrated (see \citep[Lemma 1.3.5]{moi}) or in other words that the conclusion of Definition \ref{desintegration} holds for all irreducible closed invariant subvariety of $(\mathbb{A}^2,v) ^n$  which project generically on all factors.

Let $Z \subset (\mathbb{A}^2,v)^n$ be an arbitrary irreducible closed invariant subvariety of $(\mathbb{A}^2,v)^n$. Up to a permutation of the copies of $\mathbb{A}^2$, we can assume that $Z$ projects generically exactly on the $k$ first copies of $\mathbb{A}^2$. Let $k < l \leq n$. Since the Zariski-closure of the projection of an invariant subvariety is invariant, Theorem \ref{Ilyashenko} ensures that the projection of $Z$ on the $l$-th copy of $\mathbb{A}^2$ is a zero $x_l$ of the vector field $v$. It follows that:
$$Z = Z_1 \times \lbrace  x_l \rbrace \times \ldots \times \lbrace  x_n \rbrace.$$  
where $Z_1 \subset (\mathbb{A}^2,v)^k$ is an invariant subvariety projecting generically on all factors. Using that $(\mathbb{A}^2,v)$ is generically disintegrated (Definition \ref{generically-disintegrated}), we can now write $Z_1$ as $\bigcap_{i \neq j \leq k} \pi_{i,j}^{-1}(Z_{i,j})$ and similarly for $Z$.

Now assume that $(C_2)$ holds and consider $n$ distinct non stationary analytic solutions $\gamma_1,\ldots, \gamma_n$ of $(\mathbb{A}^2,v)$ defined on an open disk $\mathbb{D} \subset \mathbb{C}$. The Zariski closure $Z$ of the image of the analytic curve
$$ t \mapsto (\gamma_1(t),\ldots, \gamma_n(t))$$
is an irreducible invariant subvariety of $(\mathbb{A}^2,v)^n$. By disintegration, it follows that $Z$ is an irreducible component of  $\bigcap_{i \neq j \leq k} \pi_{i,j}^{-1}(Z_{i,j})$ where $Z_{i,j}$ are invariant subvarieties of $(X,v)^2$. Since $(C_2)$ holds and the solutions are non stationary, $Z_{i,j} = \Delta_{i,j}$ or $Z_{i,j} = X^2$. But  as the solutions are distinct we must have $Z_{i,j} = X^2$ and therefore $Z = X^n$. This exactly means that the $n$ analytic solutions $\gamma_1, \ldots , \gamma_n$ are $\mathbb{C}$-algebraically independent.
\end{proof}

\bibliographystyle{alpha}
\bibliography{bibliographie}

\end{document}